\documentclass[12pt,reqno]{amsart}
\usepackage{tikz,dsfont}
\usepackage{graphicx}
\usepackage{amsmath, amssymb, graphicx, mathrsfs, hyperref}
\usepackage{verbatim} 
\usepackage{tabulary}
\usepackage{enumerate}
\newtheoremstyle{dotless}{}{}{\itshape}{}{\bfseries}{}{ }{} 
\theoremstyle{dotless}

\newtheorem{theorem}{Theorem}[section]
\newtheorem{lemma}[theorem]{Lemma}

\newtheorem{corollary}[theorem]{Corollary}
\newtheorem{defn}[theorem]{Definition}

\begin{document}

\pagenumbering{arabic} \setcounter{page}{1}

\title{A lower bound for the least prime in an arithmetic progression}

\author[J. Li]{Junxian Li}  
\address
{Univeristy of Illinois Urbana-Champaign,
	Urbana, Illinois 61801}\email
{jli135@illinois.edu} 
\author[K. Pratt]{Kyle Pratt}\address
{University of Illinois Urbana-Champaign,
	Urbana, Illinois 61801}
\email{kpratt4@illinois.edu}
\author[G. Shakan]{George Shakan}
\address
{University of Illinois Urbana-Champaign,
	Urbana, Illinois 61801}\email{george.shakan@gmail.com}

\maketitle

%2010 Mathematics Subject Classification: 11B13, 05B10, 11B30, Key words: additive combinatorics, sumset estimates.

\begin{abstract}
Fix $k$ a positive integer, and let $\ell$ be coprime to $k$. Let $p(k,\ell)$ denote the smallest prime equivalent to $\ell \pmod{k}$, and set $P(k)$ to be the maximum of all the $p(k,\ell)$. We seek lower bounds for $P(k)$. In particular, we show that for almost every $k$ one has $P(k) \gg \phi(k) \log k \log_2 k \log_4 k / \log_3 k,$ answering a question of Ford, Green, Konyangin, Maynard, and Tao. We rely on their recent work on large gaps between primes. Our main new idea is to use sieve weights to capture not only primes, but also small multiples of primes. We also give a heuristic which suggests that $$\liminf_{k} \frac{P(k)}{ \phi(k) \log^2 k} = 1.$$
\end{abstract}

\tableofcontents

\section{Introduction}

Fix a positive integer $k$ and let $\ell$ be coprime to $k$. Let $p(k,\ell)$ denote the smallest prime equivalent to $\ell$ modulo $k$, and define
\begin{align*}
P(k) := \max_{ (\ell,k) = 1} p(k,\ell).
\end{align*}

Linnik \cite{Li} proved the remarkable upper bound $P(k) \ll k^L$, where $L>0$ is a fixed constant. Subsequent authors improved upon the value of $L$, including Chen \cite{Ch}, Graham \cite{Gr}, Heath-Brown \cite{Hb}, Jutila \cite{Ju}, Pan \cite{Pa}, and Wang \cite{Wa}. Recently Xylouris \cite{Xy} showed that $L \leq 5.18$, following a method of Heath-Brown. Chowla \cite{Cho} observed that the Generalized Riemann Hypothesis implies $P(k) \ll_\epsilon k^{2+\epsilon}$ for any fixed $\epsilon > 0$, and conjectured that $P(k) \ll_\epsilon k^{1+\epsilon}$. In Section \ref{Heuristic section} we provide a heuristic which suggests a more precise estimate for $P(k)$.

Less work has been done on lower bounds for $P(k)$. Here the aim is to improve upon the lower bound $P(k) \geq (1 + o(1))\phi(k)\log k$, which is a consequence of the prime number theorem. Let $\log_n x$ denote the $n$-iterated logarithm ($\log_2 x = \log \log x$, $\log_{n+1} x = \log (\log_n x)$). Prachar \cite{Pra} and Schinzel \cite{Sch} showed that for each $\ell$ there are infinitely many $k$ with
\begin{align*}
p'(k,\ell) &\gg k \log k \log_2 k \frac{\log_4 k}{(\log_3 k)^2},
\end{align*}
where $p'(k,\ell)$ is the first prime $q > k$ with $q \equiv \ell \pmod{k}$. Wagstaff \cite{Wag1} showed a similar result for prime $k$.

It is very likely that $P(k)/(\phi(k)\log k)$ tends to infinity as $k$ tends to infinity (see Section \ref{Heuristic section} below). A modification of the argument of Hensley and Richards \cite{HR} shows that $P(k)/(\phi(k) \log k)$ tends to infinity for prime $k$. Pomerance \cite{Po} made a significant contribution when he showed that $P(k)/(\phi(k)\log k)$ tends to infinity for almost every $k$. Specifically, let $Q$ be the set of integers $k$ with more than $\exp (\log_2 k /\log_3 k)$ distinct prime factors. Pomerance showed that
\begin{align}\label{Pom thm}
P(k) \geq (e^\gamma +o(1)) \phi(k) \log k \log_2 k \frac{\log_4 k}{(\log_3 k)^2}
\end{align}
for every $k \not \in Q$.
Granville and Pomerance \cite{Gra} later showed there are infinitely many arithmetic progressions $\ell \pmod{k}$ such that
\begin{align*}
p(k,\ell) \geq (2+o(1)) k \log k \log_2 k \frac{\log_4 k}{(\log_3 k)^2}.
\end{align*}

Our own improvement to the lower bound for $P(k)$ builds upon the methods of Pomerance \cite{Po}. His idea was to construct a long interval $I$ of composite integers, $I = \{a,a+1, \ldots , a+ n\}$, and to consider $k \cdot I + t$ for an appropriately chosen $t$ which is coprime to $k$.

Making use of methods developed for studying large gaps between consecutive primes \cite{Fo}, we prove the following theorem.

    \begin{theorem}{\label{thm}}
    Given $\epsilon>0$, there exists $k_0(\epsilon)$ such that for all integers $k>k_0(\epsilon)$  with no more than $\exp((\frac{1}{2}-\epsilon)\log_2k\log _4 k/\log_3  k)$ distinct prime factors, we have
    	$$P(k)\gg\phi(k) \log k \log _2 k \log _4 k/\log_3 k.$$	
The implied constant is effective.
   \end{theorem}
   
We remark that the hypothesis in \eqref{Pom thm} on the number of prime factors of $k$ may be relaxed slightly, as in Theorem \ref{thm}. 

Let $z(k) =\exp((\frac{1}{2}-\epsilon)\log_2k\log _4 k/\log_3  k)$. We note that the set of $k$ which satisfy the hypothesis of Theorem \ref{thm} has density one in the natural numbers. Indeed, by an elementary bound for the sum of the divisor function, we have
\begin{align*}
\frac{1}{N} \sum_{k \leq N} 1_{\omega(k) \geq z(k)}(k) &\leq O(1/\sqrt{N}) +2^{-z(\sqrt{N})}\frac{1}{N} \sum_{\sqrt{N} < k \leq N} d(k) \ll_A \frac{1}{\log^A(N)},
\end{align*} 
for any $A>0$. Note that most $k$ have about $\log_2 k $ distinct prime factors, which is much smaller than $z(k)$.

Our main new ingredient in the proof of Theorem \ref{thm} is the use of the prime-detecting sieves of Maynard-Tao, first introduced in \cite{Ma}. Our proof of Theorem \ref{thm} follows the work of Ford, Green, Konyagin, Maynard, and Tao \cite{Fo} on large gaps between primes. Their strategy relies on sieving an interval with residue classes $a_p \pmod{p}$, building on the method of Westzynthius \cite{Wes}, as modified by Erd\H{o}s \cite{Erd} and Rankin \cite{Ran}. 

After some preliminary work, the authors of \cite{Fo} use the Maynard-Tao sieve weights to find residue classes that cover many primes simultaneously. Our approach is the same, but a complication arises in that we may only sieve with primes that are coprime to $k$. At a crucial part of the argument, we use each residue class to sieve primes and small multiples of primes, as opposed to only primes as in \cite{Fo}. We accomplish this by modifying the Maynard-Tao weights from something like $$\left(\sum_{d_i | n+ h_i} \lambda_{d_1 , \ldots , d_r}\right)^2,$$ to $$\left(\sum_{\substack{d_i | n+ h_i \\ (d_i , M) = 1}} \lambda_{d_1 , \ldots , d_r}\right)^2,$$ where $M$ is a product of very small prime divisors of $k$.

\section{Heuristics supported by data for $k \leq 10^6$}\label{Heuristic section}
   
In this section we develop a heuristic that suggests
\begin{align*}
\liminf_k \frac{P(k)}{\phi(k) \log^2 k}=1, \ \ \ \ \ \ \ \limsup_k \frac{P(k)}{\phi(k) \log^2 k}=2.
\end{align*}
We interpret the process of finding a prime in residue classes as a variant of the coupon collector problem, where the coupons are the residue classes coprime to $k$ and we collect a coupon as soon as we find a prime in that residue class. The heuristic is based on standard results from the theory of probability. We also remark that the authors of \cite{Gra} conjecture that $P(k) \gg \phi(k) \log^2 k$ for all $k$.

For a fixed $k \in \mathbb{N}$, let $m_k$ be a parameter to be chosen later. Let $p_{n}$ denote the $n^{th}$ prime and $\{a_1 , \ldots , a_{\phi(k)}\}$ be the full set of reduced residue classes modulo $k$. For $1 \leq j \leq \phi(k)$, define $E_{j}$ to be the event that $p_1 , \ldots , p_{m_k} \not\equiv a_j \pmod{k}$. The $E$ in $E_j$ can be thought of as being shorthand for ``empty," i.e. the set of the first $m_k$ primes equivalent to  $a_j$ modulo $k$ is the empty set.
Set \begin{align*}A_k := E_{1} \cup \cdots \cup E_{{\phi(k)}}.\end{align*} Thus $A_k$ represents the event that $P(k) > p_{m_k}$. Our heuristic relies on the following three assumptions. We assume that the residue classes $p_1 , \ldots , p_n$ are fixed and describe the distribution of the residue class for $p_{n+1}$:

\begin{itemize}
\item[(i)] For any $i \leq n$ such that $p_{n+1} - p_i < k$, we require that $p_{n+1}$ is in a different residue class than $p_i$ modulo $k$,
\item[(ii)] The residue class for $p_{n+1}$ is distributed uniformly from the remaining residue classes; the ones not eliminated in part (i),
\item[(iii)] The events $A_k$ are pairwise independent for all prime $k$.
\end{itemize}

Condition (i) is meant to model the basic fact that two primes that are close to each other must lie in distinct residue classes. We remark here that if we simply assumed that the residue classes modulo $k$ for each prime were independent and uniform, Lemma \ref{heur} below would remain unchanged. 

Thus, assumptions (i) and (ii) imply that the probability space is $$\Omega_k = \{(x_1,x_2,\dots,x_{m_k}) \in \{1, \ldots , \phi(k)\}^{m_k} : x_i\not=x_j \text{ if } |p_i-p_j| < k\},$$ equipped with the uniform probability measure. To understand $\liminf_k P(k)$ and $\limsup_k P(k)$ we consider $\prod_{k=1}^{\infty} \Omega_k$ equipped with the probability measure guaranteed by Kolmogorov's extension theorem (see Theorem 2.1.14 of \cite{Du}). We remark that some care must be taken with assumption (iii). For instance, it is not reasonable to assume that $A_k$ and $A_{2k}$ are independent.

We set $\pi_t :=|\{j < t : p_t-p_j<k\}|$. We compute the following probabilities exactly using conditional probability, induction, and, most importantly, assumptions (i) and (ii):
\begin{align*}
\mathbb{P}(E_i) &=\prod_{t\leq m_k} \left(1-\frac{1}{\phi(k)-\pi_t}\right), \\
\mathbb{P}(E_i\cap E_j) &=\prod_{t\leq m_k} \left(1-\frac{2}{\phi(k)-\pi_t}\right) \ \ \ (i \neq j).
\end{align*}
Note that the Brun-Titchmarsh inequality implies $\pi_t\ll \frac{k}{\log k}$, while $\phi(k)\gg\frac{k}{\log\log k}$, and so 
\begin{align*}
\mathbb{P}(E_i)&=\exp\left(-\frac{m_k}{\phi(k)}(1+o(1))\right), \\
\mathbb{P}(E_i\cap E_j)&=\exp\left(-\frac{m_k}{\phi(k)}(2+o(1))\right) \ \ \ (i \neq j).
\end{align*}
We remark that the same estimates would hold if we had assumed that the residue classes of the primes were independent and uniformly distributed. 
\begin{lemma}[Probabilistic heuristic]\label{heur} Fix $0 < \epsilon < 1/2$ and assume (i), (ii) and (iii) above. Then 

\begin{align*}\mathbb{P}(P(k) \geq p_{m_k} {\rm \ occurs \ infinitely \ often}) = 0\ \  {\rm if } \ m_k = \lceil (2+\epsilon)\phi(k)\log \phi(k) \rceil, \\ 
\mathbb{P}(P(k) \geq p_{m_k} {\rm \ occurs \ infinitely \ often})= 1 \ \ {\rm if } \ m_k = \lfloor(2-\epsilon)\phi(k)\log \phi(k) \rfloor, \\ 
\mathbb{P}(P(k) \leq p_{m_k} {\rm \ occurs \ infinitely \ often})= 1 \ \ {\rm if } \ m_k = \lceil (1+\epsilon)\phi(k)\log \phi(k) \rceil, \\ 
\mathbb{P}(P(k) \leq p_{m_k} {\rm \ occurs \ infinitely \ often}) = 0 \ \ {\rm if } \ m_k = \lfloor (1-\epsilon)\phi(k)\log \phi(k) \rfloor. \\\end{align*}
\end{lemma}

\begin{proof}
We will use the first and second Borel-Cantelli lemmas, which can be found in any graduate text in probability (for instance, section 2.3 of \cite{Du}).

    By the first Bonferroni inequality, we have 
    $$\mathbb{P}(A_k) \leq \sum_{j=1}^{\phi(k)} \mathbb{P}(E_{j}) = \phi(k)\prod_{t\leq m_k} \left(1-\frac{1}{\phi(k)-\pi_t}\right).$$ 
    If $m_k =\lceil(2+\epsilon)\phi(k)\log \phi(k)\rceil$ then we have
    $$\sum_{k=1}^\infty\mathbb{P}(A_k) \leq \sum_{k=1}^\infty\phi(k)\exp\left(-\frac{m_k}{\phi(k)}(1+o(1))\right)\leq \sum_{k=1}^\infty\frac{1}{\phi(k)^{(1+\epsilon)(1 + o(1))}}<\infty.$$
    The second Bonferroni inequality implies 
$$\mathbb{P}(A_k) \geq \phi(k)\prod_{t\leq m_k} \left(1-\frac{1}{\phi(k)-\pi_t}\right) - {\phi(k) \choose 2}\prod_{t\leq m_k} \left(1-\frac{2}{\phi(k)-\pi_t}\right).$$
If $m_k=\lfloor(2-\epsilon)\phi(k)\log\phi(k)\rfloor$ we obtain 
\begin{align*}
\sum_{k \text{ prime}}\mathbb{P}(A_k) &\geq \sum_{k \text{ prime}} \left[\phi(k)\exp\left(-\frac{m_k}{\phi(k)}(1+o(1))\right)-{\phi(k)\choose 2}\exp\left(-\frac{m_k}{\phi(k)}(2+o(1))\right) \right]\\
&\geq \sum_{k \ \text{prime}}\left( \frac{1}{\phi(k)^{(1-\epsilon)(1 + o(1))}}-\frac{1}{\phi(k)^{(2-2\epsilon)(1+o(1))}} \right) =\infty.
\end{align*}

In conclusion, we have
    \begin{align*}
    \left\{\begin{array}{lr}
    	 \sum_{k=1}^\infty  \ \  \ \mathbb{P}(A_k)<\infty & \text{if } m_k =\lceil(2+\epsilon)\phi(k)\log \phi(k)\rceil,\\
    	\sum_{k {\rm \ prime}}\mathbb{P}(A_k)=\infty & \text{if } m_k =\lfloor (2-\epsilon)\phi(k)\log \phi(k)\rfloor.
    \end{array}\right.
    \end{align*}
%We remark that in checking the above inequalities, one can see that the same result would hold if we had assumed that the residue classes for the primes were chosen uniformly and independently.
 When $m_k=\lceil(2+\epsilon)\phi(k)\log \phi(k)\rceil$, the first Borel-Cantelli lemma implies that $A_k$ occurs infinitely often with probability 0, giving the first claim. When $m_k=\lfloor (2-\epsilon)\phi(k)\log \phi(k)\rfloor$, the second Borel-Cantelli lemma, along with assumption (iii), establishes the second claim.

We now assume $m_k =\lceil (1+\epsilon)\phi(k) \log \phi(k)\rceil $. Note that the event $P(k) \leq p_{m_k}$ is precisely $A_k^c$ and $$\mathbb{P}(A_k^c) \geq 1 - \phi(k)\prod_{t\leq m_k} \left(1-\frac{1}{\phi(k)-\pi_t}\right) \geq 1 - 1/\phi(k)^{\epsilon(1 + o(1))} = 1-  o(1).$$ Now the third claim follows from the second Borel-Cantelli lemma along with assumption (iii).

It remains to show the fourth claim. Assume $m_k =\lfloor (1- \epsilon) \phi(k) \log \phi(k)\rfloor$. Inclusion-exclusion is no longer useful as the first few summands are too large. The new idea is to show that the events $E_{1}^c , \ldots , E_{{\phi(k)}}^c$ are negatively correlated, that is \begin{equation}\label{corr} \mathbb{P}(E_{1}^c \cap \ldots \cap E_{{\phi(k)}}^c) \leq \mathbb{P}(E_{1}^c) \cdots \mathbb{P}(E_{{\phi(k)}}^c).\end{equation}  Intuitively, if the first few coupons are known to be collected, then it is slightly less likely that the next coupon will also be collected. 
By induction, it is enough to show, for $1 \leq t \leq \phi(k) -1$, that $$\mathbb{P}(E_{1}^c \cap \ldots \cap E_{{t+1}}^c) \leq \mathbb{P}(E_{1}^c \cap \ldots \cap E_{{t}}^c)\mathbb{P}(E_{{t+1}}^c).$$ This is equivalent to  $$\mathbb{P}(E_{1}^c \cap \ldots \cap E_{{t}}^c | E_{{t+1}}^c) \leq \mathbb{P}(E_{1}^c \cap \ldots \cap E_{{t}}^c).$$ Using that, for any nonempty events $C$ and $D$, one has $\mathbb{P}(C | D^c) \leq \mathbb{P}(C)$ if and only if $\mathbb{P}(C | D) \geq \mathbb{P}(C)$, this is equivalent to $$\mathbb{P}(E_{1}^c \cap \ldots \cap E_{{t}}^c | E_{{t+1}}) \geq \mathbb{P}(E_{1}^c \cap \cdots \cap E_{{t}}^c).$$ 

For any $\mathcal{C}\subset \{p_1 , \ldots , p_{m_k}\}$, let $F(\mathcal{C})$ be the event that $\mathcal{C}$ is the set of primes $p \leq p_{m_k}$ congruent to $a_{t+1} \pmod{k}$. Observe that conditioning on $F(\mathcal{C})$ is equivalent to removing one residue class and primes in $\mathcal{C}$  from the probability space. Then $E_{t+1}$ corresponds to the case $\mathcal{C}=\emptyset$. Since $\mathbb{P}(E_{1}^c \cap \ldots \cap E_{t}^c|F(\mathcal{C}))$ is monotone decreasing in $|\mathcal{C}|$, we have $\mathbb{P}(E_{1}^c \cap \ldots \cap E_{t}^c |F(\mathcal{C})) \leq \mathbb{P}(E_{1}^c \cap \ldots \cap E_{t}^c | E_{{t+1}})$. Then \begin{align*} \mathbb{P}(E_{1}^c & \cap \ldots \cap E_{t}^c) = \sum_{\mathcal{C} \subset \{p_1 , \ldots , p_{m_k}\}} \mathbb{P}(E_{1}^c \cap \ldots \cap E_{t}^c | F(\mathcal{C})) \mathbb{P}(F(\mathcal{C})) \\
&\leq \mathbb{P}(E_{1}^c \cap \ldots \cap E_{{t}}^c | E_{{t+1}}) \left( \sum_{\mathcal{C} } \mathbb{P}(F(\mathcal{C}))\right) = \mathbb{P}(E_{1}^c \cap \ldots \cap E_{{t}}^c | E_{{t+1}}) . \end{align*} This shows \eqref{corr}. We also have  
\begin{equation}\label{dev} 
\mathbb{P}(E^c_{1}) \cdots \mathbb{P}(E^c_{\phi(k)}) = \left(1 - \prod_{t\leq m_k} \left(1-\frac{1}{\phi(k)-\pi_t}\right) \right)^{\phi(k)} \leq \exp \left(-\phi(k)^{\epsilon(1+o(1))}\right) \ll k^{-2}.
\end{equation} 
The final claim now follows from \eqref{corr} and the first Borel-Cantelli lemma.
\end{proof}

%with $\epsilon = 1/n$ for each $n \in \mathbb{N}$, we see that $\liminf_k P(k) = p_{m_k}$ with probability 1, where $m_k \sim  \phi(k) \log k$. By
Applying Lemma \ref{heur}, the prime number theorem and the fact that $\log \phi(k) \sim \log k$, we obtain $\liminf_k P(k) / \phi(k) \log^2 k = 1$ with probability 1. In a similar manner, it also follows from Lemma \ref{heur} that $\limsup_k P(k) / \phi(k) \log^2 k = 2$ with probability 1. 

 We remark that Wagstaff \cite{Wag2} provides a heuristic, supported by numerical data, which claims that the typical value of $P(k)$ is $\phi(k) \log^2 k$. Indeed, one could apply a variant of the weak law of large numbers as in Example 2.2.3 in \cite{Du} to get $P(k)/\phi(k) \log^2 k \to 1$ in probability. In Figure \ref{picRatio} we calculate $$P(k)/(\phi(k)\log(\phi(k))\log k),$$ for $k\leq 10^6$. Note that this quantity is very concentrated near $1$. 

We remark that the third and fourth claims of Lemma \ref{heur} are basically a medium deviation result of the coupon collector problem similar in spirit to that of Example 3.6.6 in \cite{Du}. In the notation there, $x$ is required to be fixed, but we require $x$ to be of size $-\epsilon \log \phi(k)$. From this perspective, we understand why the $\exp\left(-\phi(k)^{\epsilon(1 + o(1))}\right)$ appeared in \eqref{dev}. In Figure \ref{Distribution} we show the distribution of the quantity 
$$r_k=\frac{P(k)-\phi(k)\log \phi(k)\log P(k)}{\phi(k)\log{P(k)}},$$ for $k$ up to $10^6$ is approximately  $\mathbb{P}(r_k\leq x) = e^{-e^{c-b x}}$, a Gumbel distribution, where $b\approx 1.45$ and $c\approx -13.6/e$. This should be compared to example 3.6.6 in \cite{Du}. In fact, if $\xi_n$ are independent variables with $\mathbb{P}(\xi_n\pmod k= a_i)= \frac{1}{\phi(k)}$ for all residue classes $a_i$ coprime to $k$, then equation (2) in \cite{ErdosRenyi} implies that the waiting time for each residue class to be filled  

$$w(k):=\min\{n: {\rm there \ exists} \  t\leq n  \ {\rm such \ that } \  \xi_t\equiv a_i, \ {\rm for \ all} \ i=1,\cdots,\phi(k) \},$$ 

has the asymptotic distribution $$\lim_{k\rightarrow\infty} \mathbb{P}\left(\frac{w(k)}{\phi(k)}-\log(\phi(k))<x\right)=e^{-e^{-x}}.$$Due to assumption (i), it is not clear that $r_k$ has a limiting distribution. If $r_k$ does have a Gumbel distribution, it is not clear what the expected parameters should be. 

%We remark that in \cite{Fo2}, the Gumbel distribution also appeared in a number theoretic context. 

Our conditions (i), (ii), and (iii) are not the only reasonable simplifying assumptions one could imagine for a probabilistic model of $P(k)$. Nevertheless, the nature of the coupon collector problem is that many coupons are collected quickly and one has to wait a long time to collect the last few coupons (see the calculations in Example 2.2.3 in \cite{Du}). With any set of assumptions, we inevitably arrive at the situation where we are seeking a prime in one of very few residue classes. We are unable to think of any set of assumptions that would allow us to have any control of this part of the process. 
\begin{figure}
\includegraphics[width=\textwidth]{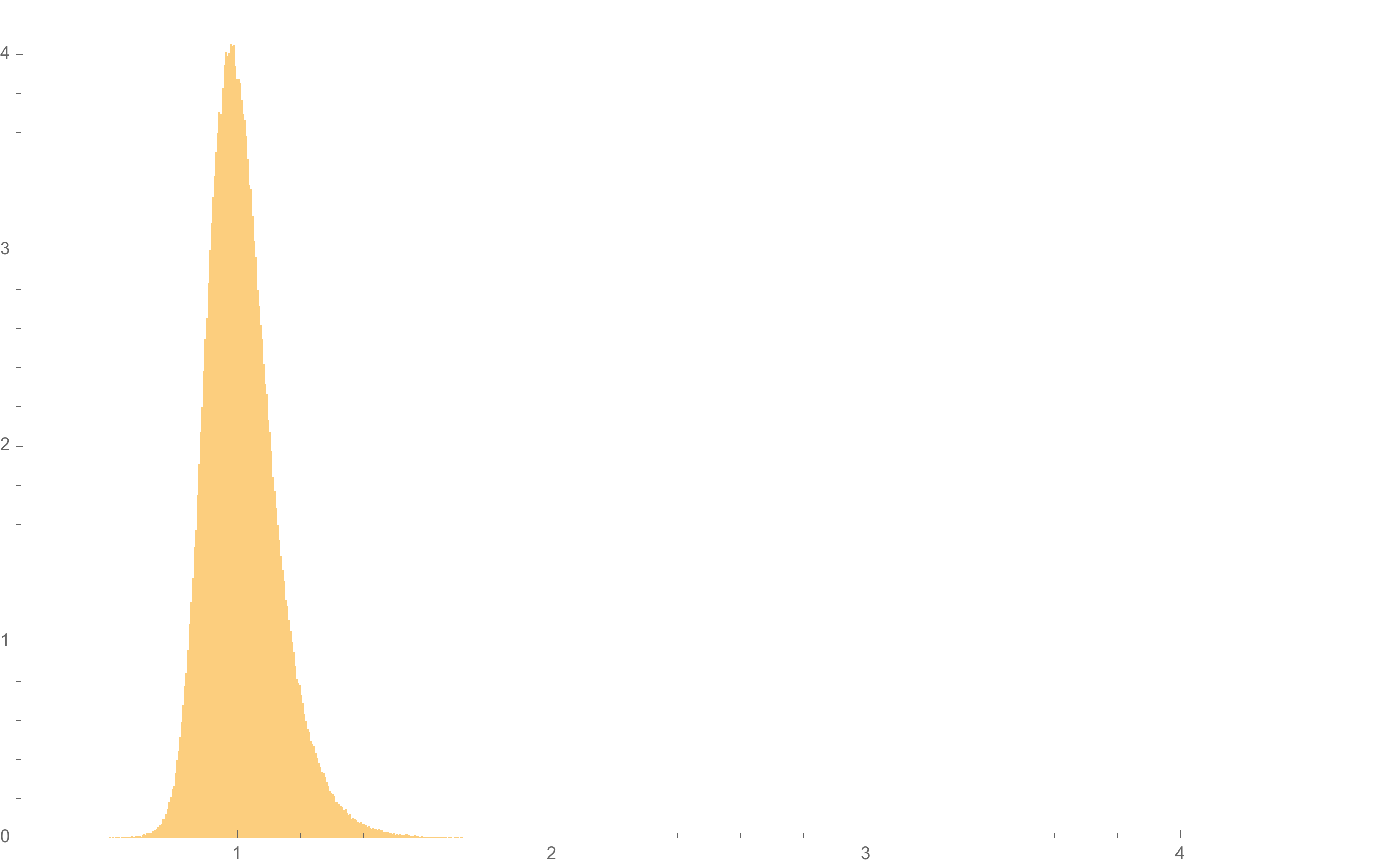}
\caption{Histogram for $P(k)/\phi(k)\log(\phi(k))\log k$ for $k\leq 10^6$}
\label{picRatio}
\end{figure}

\begin{figure}
\includegraphics[width=\textwidth]{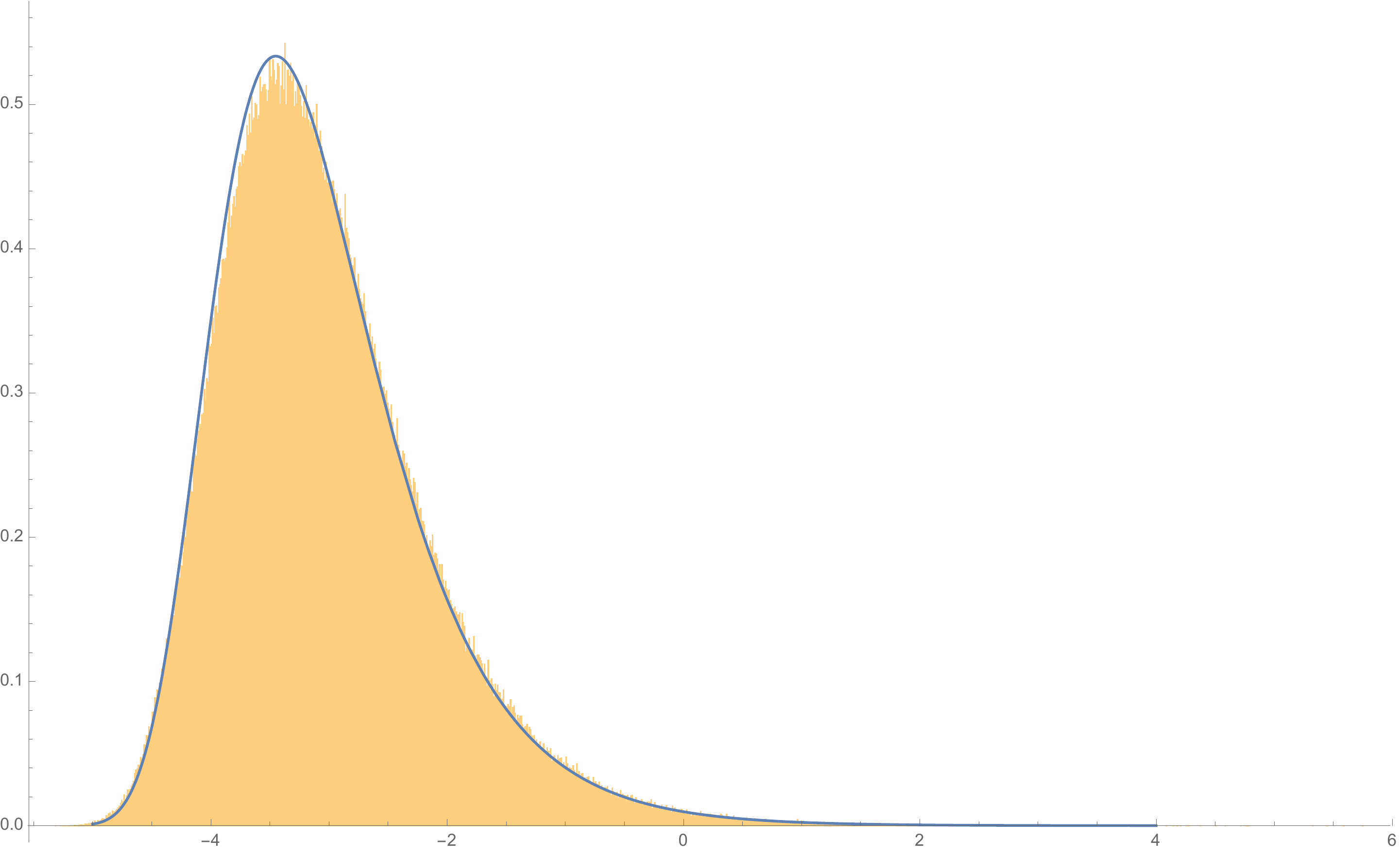}
\caption{Histogram of $r_k=\frac{P(k)-\phi(k)\log \phi(k)\log P(k)}{\phi(k)\log P(k)}$ for $k \leq 10^6$, and the density function of the distribution $e^{-e^{c-bx}}$, with $b\approx 1.45$, $c\approx -13.6/e$ .}
\label{Distribution}
\end{figure}

\section{Notation and Conventions}
For a set of primes $S$ and for each $p \in S$, let $a_p \pmod{p}$ be a residue class. We will denote this sequence of residue classes by $(a_p \pmod{p})_{p \in S}$, or simply by $(a_p)$ or even $\vec{a}$ when the meaning is clear from context.

We will use $\vec{\bold{a}}$ to denote a sequence of residue classes chosen randomly from a probability distribution.

For a positive  integer $n$, we set $P^+(n)$ to be the largest prime factor of $n$ ($P^+(1) := 1$). We also let $\phi(n)$ denote Euler's totient function. 

We say $h_1 , \ldots , h_r \in \mathbb{Z}$ is an admissible $r$-tuple if for every prime $p$ we have $\{h_1 \pmod{p} , \ldots , h_r \pmod{p}\} \neq \mathbb{Z} / p \mathbb{Z}$.
 
Let $\mathcal{L}=\{L_1,\dots,L_k\}$ be a set of distinct linear functions $L_i(n)=a_i n+b_i,1\leq i\leq k$, where $a_i,b_i$ are integers. We say $\mathcal{L}$ is admissible if $\prod_{i=1}^kL_i(n)$ has no fixed prime divisor. That is, for every prime $p$, there is an integer $n_p$ such that $\prod_{i=1}^kL_i(n_p)$ is coprime to $p$. 

We use $X = O(Y)$ to mean there exists a constant $C>0$ such that $|X| \leq CY$ throughout the domain of $X$. We write $X\ll Y$ to mean $X = O(Y)$. If the implied constant may be taken to be one, we write $X = O_\leq (Y)$. The notation $X \asymp Y$ means $X \ll Y$ and $Y \ll X$. 

We write $g(k) = o(f(k))$ if $\frac{g(k)}{f(k)} = o(1)$, and $f(k) \sim g(k)$ if $f(k) = (1+o(1))g(k)$. The notation $o(1)$ denotes a quantity that tends to zero as $k$ goes to infinity. 

From now on all implied constants may depend on $\epsilon$, and any other dependence is explicitly noted. 

\section{Outline of the proof of Theorem \ref{thm}}

We give an informal description of the proof of Theorem \ref{thm} before turning to the details in earnest. Nothing in this section will be used in later sections.

We use a theorem of Pomerance \cite{Po} to reduce to showing there exist residue classes of primes $\leq (1 - o(1)) \log k$ that cover an interval of length $y$. Here $y$ is much larger than $\log k$. The caveat is that we are not allowed to use primes which divide $k$. We choose many of these residue classes to be $\equiv 0 \pmod{p}$. We crucially use smooth number estimates to show that what remains after this first step is substantially smaller than what naive heuristics (or a sieve) would predict. It is in this step that our assumption on the number of prime divisors of $k$ is most important. The main difference from the arguments of \cite{Fo} already appears in this step, as we are left with both primes and small multiples of primes. 

Next, we choose many of the residue classes of the medium-sized primes uniformly at random. It is important that these primes are not too small, in order to show what remains after this step has nice distributional properties (see Lemma \ref{st}).

In the third step, we condition on the random residue classes chosen in the previous step and choose the residue classes for large primes $\asymp \log k$. This step is also random, but we use a modified version of the Maynard-Tao weights to create our probability distribution. What remains from the previous two steps is a sparse subset of an interval of length $y$. In general, one cannot hope to cover such a set without additional information. For instance, if what remained consisted of $L$ consecutive integers, we could only hope to cover $\asymp L/\log k$ integers with each prime. 

We use the fact that what remains after step one is typically covered by our modification of the Maynard-Tao weights and that, with high probability, what remains after step two interacts well with the Maynard-Tao weights. This would already give an improvement to Theorem 3 of \cite{Po}, but we seek to optimize our argument by utilizing a hypergraph covering lemma from \cite{Fo}. This ensures that the residue classes from this third step cover what remains almost disjointly. 

%We remark that in step 3, it is important that $y$ is much larger than $\log k$, and the reason that in our current situation this is the case is a consequence of the smooth number estimates from step one. 

In the final step, what remains is so small that we use our leftover primes, saved just for this purpose, to cover unsieved elements one at a time.

\section{First Steps towards Theorem \ref{thm}}

Our proof of Theorem \ref{thm} begins with a result due to Pomerance \cite{Po}. For $m \in \mathbb{N}$ we define Jacobsthal's function $g(m)$ to be the largest difference between consecutive integers coprime to $m$. Thus for instance, there exist $g(m)-1$ consecutive integers all of which have a prime factor in common with $m$. 
\begin{lemma}\label{gm}
Suppose $k,m$ are integers, with $0<m\leq \frac{k}{1+g(k)}$ and $(m,k)=1$. Then $$P(k)>(g(m)-1)k.$$
\end{lemma}
\begin{proof}
This is Theorem 1 of \cite{Po}.
\end{proof} Fix $\epsilon>0$, and let $$x:=(1-\epsilon)\log k.$$ We apply this lemma with 
$$m=\prod_{\substack{p \leq x \\ p \nmid k}}p.$$ 
By the prime number theorem we have $m<k^{1-\epsilon/2}$ for all large $k$. A simple sieve argument shows $g(m) \ll (\log m)^{O(1)}$, and Iwaniec \cite{Iw} showed $g(m)\ll \log^2m$. Thus, the hypotheses of Lemma \ref{gm} are satisfied for our choice of $m$ when $k$ is sufficiently large. The proof of Theorem \ref{thm} is then reduced to proving
\begin{align}\label{Lower bound for gm}
g(m) \gg \frac{\phi(k)}{k} \log k \log_2 k\frac{\log_4 k}{\log_3 k}.
\end{align}

\section{Random Construction}
Our proof of Theorem \ref{thm} closely follows the arguments of \cite{Fo}. The arguments in this section correspond to Section 4 of \cite{Fo}.

Let $y: = c_k x\log x \log_3 x/\log_2 x$,  where $c_k \in (0,1]$ is a parameter, defined below in \eqref{ck}, that satisfies $c_k \gg  \frac{\phi(k)}{k}$. Our goal is to find residue classes $(a_p \pmod{p})_{p\mid m}$ such that for every integer $n\in(x,y]$, we have $n\equiv a_p\pmod{p}$ for some $p$ dividing $m$. By the Chinese remainder theorem there exists $t \in \mathbb{N}$ such that $t\equiv -a_p\pmod{p}$ for all $p$ dividing $m$. Thus, for every $n \in (x,y]$ there exists a prime $p$ dividing $m$ such that $t + n \equiv -a_p + a_p \equiv 0 \pmod{p}$, which shows that $g(m)\geq y-x$.

Let $$z:=\exp\left({\left(\frac{1-\epsilon}{2}\right) \frac{\log x\log_3 x}{\log _2 x}}\right),$$ and consider the disjoint sets of primes
	\begin{align*}
	\mathcal{S}&:=\{s\mbox{ prime: }\log^{20}x< s \leq z, \ s\nmid k\},\\
    \mathcal{P}&:=\{p\mbox{ prime: } x/2< p \leq x, \ p\nmid k \}.
	\end{align*}
%	\begin{center}	
%\begin{tikzpicture}
%\draw (0,0)--node[below]{1} (2,0) -- node[below] {$\log^{20}x$} ++(1,0) -- node[above]{$\mathcal{S}$}++(1,0) -- node[below] {$z$}++(2,0)-- node[below]{$x/4$} ++(2,0)-- node[below]{$x/2$} ++(1,0) -- node[above]{$\mathcal{P}$}++(1,0) --  node[below]{$x$}++(1,0) -- ++(1,0) -- node[below]{$y$}  ++(4,0);
%\end{tikzpicture}
%	\end{center}
 
We choose the residue classes $(a_p \pmod{p})_{p\mid m} = (a_p)_{p|m}$ in four stages.
\begin{enumerate}
	\item[Stage 1.] Choose $a_p \equiv 0 \pmod{p}$ for the primes $p\leq \log^{20}x$ and $p\in(z,x/4]$;
	\item[Stage 2.] For each prime $s\in \mathcal{S}$, select each $\bold{a}_s\pmod{s}$ independently and uniformly at random. Let $\vec{\bold{a}}:=(\bold{a}_s\pmod{s})_{s\in\mathcal{S}}$;
	\item[Stage 3.] For each prime $p\in \mathcal{P}$, select a residue class $b_p \pmod{p}$ strategically depending on $\vec{\bold{a}} $;
	\item[Stage 4.] Select residue classes for primes in $(x/4,x/2]$ to cover the elements of $(x,y]$ left uncovered by earlier stages, matching each uncovered element with a prime and choosing residue classes accordingly.
\end{enumerate}
Hence, to prove \eqref{Lower bound for gm} it is sufficient to show that the number of elements left uncovered after the first three stages is less than $ \pi(x/2)-\pi(x/4) = (1+o(1)) \frac{x}{4\log x }$.

%\begin{lemma}\label{sifted}
%	Let $x$ and $y$ defined as before, then there are vectors $\vec{a}=(a_s\mod s)_{s\in\mathcal{S}}$ and $\vec{b}=(b_p\mod p)_{p\in\mathcal{P}}$ such that $$\#\mathcal{Q}\cap S(\vec{a}) \cap T(\vec{b})\leq \frac{x}{5\log x}$$.
%\end{lemma}
 After stage 1, what remains uncovered in $(x,y]$ falls into one of the following three sets. 
\begin{itemize}
	\item $\mathcal{ZS}:=\{n:  P^+(n)\leq z\}$,
	\item $\mathcal{ZR}:=\{n : \text{ there exists } p \mid n  \mbox{ such that } p\mid k \text{ and }z<p<x/4\}$,
	\item $\mathcal{MQ}:=\{n=mq : p\mid m \text{ implies } p\mid k\text{ and }p\leq 4y/x; q \mbox{ is a prime in } (x/4,y]\ \}$.
\end{itemize}
	\begin{center}	
\begin{tikzpicture}
\draw[->] (0,0)--node[above]{$0 \pmod {p}$}(1,0) -- (2,0) -- node[below] {$\log^{20}x$} ++(1,0) -- node[above]{$\mathcal{S}$}++(1,0) -- node[below] {$z$}++(2,0)--node[above]{$0 \pmod {p}$}++(1,0)-- node[below]{$x/4$} ++(2,0)-- node[above]{Stage 4}++(1,0)--node[below]{$x/2$} ++(1,0) -- node[above]{$\mathcal{P}$}++(1,0) --  node[below]{$x$}++(1,0) --node[above]{$\mathcal{MQ}$}++(1,0)  -- node[below]{$y$}  ++(1,0);
\end{tikzpicture}
	\end{center}

We show that $\#\mathcal{ZS}$ and $\#\mathcal{ZR}$ are small enough to be easily covered in stage 4. Rankin's method \cite{Ran} for estimating smooth numbers, which can be found in \cite{Notes1}, for instance, gives 
\begin{align*}
\#\mathcal{ZS}\leq y e^{-(1+o(1))\frac{\log y}{\log z}\log(\frac{\log y}{\log z}) }
= \frac{y}{\log^{2/(1-\epsilon)+o(1)} x }=o\left(\frac{x}{\log x}\right).
\end{align*} 
The assumption that $\omega(k) \leq \exp((\frac{1}{2}-\epsilon)\log _2k\log_4k/\log_3 k)$ implies
\begin{align*}
\#\mathcal{ZR}&\leq \sum_{p>z,p\mid k}\frac{y}{p} \leq \frac{y}{z}\exp\left(\left(\frac{1}{2}-\epsilon\right)\log_2 k\frac{\log _4 k}{\log_3 k}\right)\\
&\leq y\exp\left(-\frac{\epsilon}{3}\frac{\log _3 x}{\log_2 x}\log x\right) = o\left(\frac{x}{\log x}\right).
\end{align*} 
 
For residue classes $\vec{a}=(a_s\pmod{s})_{s\in\mathcal{S}}$ and $\vec{b}=(b_p\pmod{p})_{p\in\mathcal{P}}$, define the sifted sets
\begin{align*}	
S(\vec{a}):&=\{n\in\mathbb{Z}:\ n\not\equiv a_s\pmod{s}, \mbox{ for all $s\in \mathcal{S}$}\},\\
T(\vec{b}):&=\{n\in\mathbb{Z}:\ n\not\equiv b_p\pmod{p}, \mbox{ for all $p\in \mathcal{P}$}\}.
\end{align*}
Thus, it is enough to show there exist $\vec{a}$ and $\vec{b}$ such that
\begin{equation}\label{rdc}
\#(\mathcal{MQ}\cap S(\vec{a}) \cap T(\vec{b}))\leq \frac{x}{6\log x}.
\end{equation}
Define 
\begin{align*}
M:=\prod_{\substack{p\leq \log x \\ p\mid k}}p, \ \ \ \ \ \ \ \  \ \ \ \kappa:=\prod_{\substack{\log^{20}x<p\leq z\\p\mid k}}p.
\end{align*}
We note here that it is important for the definition of $M$ that $y/x\ll \log x$. Set
\begin{equation}\label{ck} c_k : = c \frac{\phi(M)}{M} \frac{\phi(\kappa)}{\kappa}, 
\end{equation} 
where $c > 0$ is some small fixed constant. Let 
\begin{align}\label{sigma}
\sigma:=\prod_{s\in\mathcal{S}}\left(1-\frac{1}{s}\right)\sim \frac{40}{1-\epsilon}\frac{\kappa}{\phi(\kappa)}\frac{(\log_2x)^2}{\log x\log_3 x}.
\end{align}
%$$\sigma:=\prod_{s\in\mathcal{S}}(1-\frac{1}{s})\sim \frac{\kappa}{\phi(\kappa)}\frac{\log(\log^{20}x)}{\log z}\sim \frac{\kappa}{\phi(\kappa)}\frac{(40+20\epsilon)(\log_2x)^2}{\log x\log_3 x}.$$
%\begin{center}
%	\begin{tabular}{ | l | c | l | p{5cm} |}
%		\hline
%		set & definition & expected size\\ \hline
%		$\mathcal{S}$&$\{\mbox{s prime: }\log^{20}x< s \leq z, \ s\nmid k\}$&\\ \hline
%		$\mathcal{P}$&$\{\mbox{p prime: } x/2< p \leq x, \ p\nmid k \}$&$\sim\frac{x}{2\log x}$\\ \hline
%	   $\mathcal{Q}$&$\{\mbox{q prime: } x< q \leq y \}$&$\sim\frac{y}{\log x}$ \\ \hline
%		$\mathcal{MQ}$ &$\{n=mq : p\mid m \implies p\mid M, \mbox{q is a prime in } (x/4,y]\}$ &  $\leq\frac{M}{\phi(M)}\frac{y}{\log x}$ \\ \hline
%			$\mathcal{MQ}\cap S(\vec{a})$ &$\{n=mq,n\in S(\vec{a}) : p\mid m \implies p\mid M, \mbox{q is a prime in } (x/4,y]\}$ &  $\leq\frac{M}{\phi(M)}\frac{\sigma y}{\log x}$ \\ \hline
%	\end{tabular}
%\end{center}

%\begin{align*}
%\sigma \frac{M}{\phi(M)}y&\sim \kappa\frac{M}{\phi(M)}\frac{(40+20\epsilon)(\log_2x)^2}{\log x\log_3x}y\\
%&\sim (40+20\epsilon)c\prod_{\substack{\log^{20}x<p\leq z\\p\mid k}}(1-\frac{1}{p})^{-1}\prod_{\substack{p\leq \frac{4y}{x}\\p\mid k}}(1-\frac{1}{p})^{-1}\frac{\phi(k)}{k}x\log_2x\\
%&\leq (40+21\epsilon)cx\log_2x
%\end{align*}

In the second stage $\vec{\bold{a}}$ is chosen randomly, and with probability $1-o(1)$ the set $\mathcal{MQ}\cap S(\vec{\bold{a}})$ has the expected size $\sigma\#\mathcal{MQ}$. We then want to use each residue class $b_p \pmod{p}$ to cover many elements of $\mathcal{MQ}\cap S(\vec{\bold{a}})$.

For the method to work, it is crucial that we choose $(b_p)$ depending on $(a_s)$, since we want $(b_p)$ to sieve out elements of $\mathcal{MQ}$ left uncovered by $(a_s)$. The next lemma is the main tool that eventually allows us to do this. 
 
\begin{lemma}\label{rdm}
	Let $x,y$ be as above. Then there is a quantity $C$ with 
	$$C\asymp \frac{1}{c},$$ with the implied constants independent of $c$, a tuple of positive integers $(h_1,\dots,h_r)$ with $r\leq \sqrt{\log x}$ and some way to choose random vectors $\vec{\bold{a}}=(\bold{a}_s\pmod{s})_{s\in \mathcal{S}}$ and $\vec{\bold{n}}=(\bold{n}_p\pmod{p})_{p\in \mathcal{P}}$ such that there exist $\mathcal{P}(\vec{a})\subset \mathcal{P}$ with $\#\mathcal{P}(\vec{a})=\left(1+O\left(\frac{1}{\log^{3}x}\right)\right)\#\mathcal{P}$. \begin{itemize}
		\item For every $\vec{a}$ and all $p \in \mathcal{P}(\vec{a})$,
		\begin{equation} \label{rdm0}\mathbb{P}(q\in \bold{e}_p(\vec{a})|\vec{\bold{a}}=\vec{a})\leq x^{-1/2-1/10},\end{equation} where $\bold{e}_p(\vec{a}):=\{\bold{n}_p+h_ip:1\leq i\leq r\}\cap\mathcal{MQ}\cap S(\vec{a}).$ 
		\item With probability $1-o(1)$, 
	\begin{equation}\label{rdm1}\#(\mathcal{MQ}\cap S(\vec{\bold{a}})) \leq c\frac{41}{1-\epsilon}\frac{x}{\log x}\log_2x.\end{equation}
		\item Call an element $\vec{a}$ \emph{good} if, for all but at most $O(\frac{1}{\log_2^2 x}\#(\mathcal{MQ}\cap S(\vec{a})))$ elements of $\mathcal{MQ}\cap S(\vec{a})$,
		\begin{equation}\label{rdm2}\sum_{p\in \mathcal{P}(\vec{a})}\mathbb{P}(q\in \bold{e}_p(\vec{a})|\vec{\bold{a}}=\vec{a})=C+O_\leq\left(\frac{1}{\log_2^2x}\right).
		\end{equation}
		Then $\vec{\bold{a}}$ is good with probability $1-o(1)$.
	\end{itemize}
\end{lemma}
We use the following lemma, which is Corollary 4 of \cite{Fo}, to ensure we may find residue classes $(b_p \pmod{p})$ that sieve almost disjointly.
\begin{lemma}\label{Covering}
	Let $x\rightarrow\infty$. Let $\mathcal{P'}, \mathcal{Q'}$ be sets with $\#\mathcal{P'}\leq x$ and $\#\mathcal{Q'}>(\log_2x)^3$. For each $p\in \mathcal{P'}$, let $\bold{e}_p$ be a random subset of $\mathcal{Q'}$ satisfying the size bound 
	$$\#\bold{e}_p\leq r=O\left(\frac{\log x\log_3x}{\log_2^2x}\right).$$
	Assume the following:
	\begin{itemize}
		\item (Sparsity) For all $p\in \mathcal{P'}$ and $q\in\mathcal{Q'}$, $$\mathbb{P}(q\in\bold{e}_p)\leq x^{-1/2-1/10}.$$
		\item(Uniform covering) For all but at most $\frac{1}{\log_2^2 x}\#\mathcal{Q'}$ elements of $\mathcal{Q}'$, we have $$\sum_{p\in\mathcal{P'}}\mathbb{P}(q\in\bold{e}_p)=C+O_\leq\left(\frac{1}{\log_2^2x}\right),$$ for some quantity $C$ independent of $q$ satisfying $\frac{5}{4}\log 5\leq C\ll 1$.
		\item (Small codegrees) For any distinct $q_1,q_2\in \mathcal{Q}'$, $$\sum_{p\in\mathcal{P'}}\mathbb{P}(q_1,q_2\in\bold{e}_p)\leq x^{-1/20}.$$
	\end{itemize} 
		Then for any positive integer $m\leq \frac{\log_3x}{\log 5}$, we can find random sets $\bold{e'_p}\subset\mathcal{Q'}$ for each $p\in \mathcal{P'}$ such that 
		$$\#\{q\in\mathcal{Q'}:q\not\in\bold{e'_p} \mbox{  for all } p\in\mathcal{P'}\}\sim 5^{-m}\#\mathcal{Q'}$$ with probability $1-o(1)$. 
%		More generally, for any $\mathcal{Q''}\subset\mathcal{Q'}$ with cardinality at least $\#\mathcal{Q'}/\sqrt{\log_2x}$, $$\#\{q\in\mathcal{Q''}:q\not\in\bold{e'_p} \mbox{  for all } p\in\mathcal{P'}\}\sim 5^{-m}\#\mathcal{Q''}$$ with probability $1-o(1)$. 
		The decay rate in the $o(1)$ and $\sim$ notation are uniform in $\mathcal{P'}$ and $\mathcal{Q'}$.
\end{lemma}

Now we show how Lemmas \ref{rdm} and \ref{Covering} imply \eqref{Lower bound for gm}. Let $0<c<1/2$ be small enough so that $\frac{5}{4}\log 5\leq C$. Take $m=\lfloor\frac{\log_3x}{\log 5}\rfloor.$ Let $\vec{a}$ be a vector such that \eqref{rdm1} and \eqref{rdm2} hold. We use Lemma \ref{Covering} with $\mathcal{P'}=\mathcal{P}(\vec{a})$ and $\mathcal{Q'}=\mathcal{MQ}\cap S(\vec{a})$ for the random variable $\bold{n}_p$ conditioned to $\vec{\bold{a}}=\vec{a}$. Then \eqref{rdm0} implies the sparsity condition, and \eqref{rdm2} gives the uniform covering condition. 

Let $q_1,q_2$ be distinct elements of $\mathcal{MQ}\cap S(\vec{a})$. If $q_1,q_2\in \bold{e}_p(\vec{a})$ then $p\mid q_1-q_2$. Since $q_1-q_2=O(x\log x)$ and $p\gg x$, there is at most one $p_0\in \mathcal{P'}$ dividing $q_1-q_2$, which implies
$$\sum_{p\in\mathcal{P'}}\mathbb{P}(q_1,q_2\in\bold{e}_p(\vec{a}))\leq\mathbb{P}(q_1\in\bold{e}_{p_0}(\vec{a}))\leq x^{-1/2-1/10}.$$
This gives the small codegrees condition.

By Lemma \ref{Covering}, there exist variables $\bold{e'}_p(\vec{a})$ satisfying 
\begin{align}\label{good eps}
\#\{q\in\mathcal{MQ}\cap S(\vec{a}):q\not\in\bold{e'_p} \mbox{  for all } p\in\mathcal{P'}\}\sim 5^{-m}\#(\mathcal{MQ}\cap S(\vec{a}))\ll \frac{cx}{\log x}
\end{align}
with probability $1-o(1)$, the implied constant being absolute. Since $\bold{e'_p}=\{\bold{n'}_p+h_ip:1\leq i\leq r\}\cap \mathcal{MQ}\cap S(\vec{a})$, for some random integer $\bold{n'}_p$, we can choose some $n'_p$  so that \eqref{good eps} holds. For each $p\in\mathcal{P}(\vec{a})$ set $b_p\equiv n'_p \pmod{p}$, and for each $p\not\in\mathcal{P}(\vec{a})$ set $b_p \equiv 0 \pmod{p}$. Taking $c$ sufficiently small gives \eqref{rdc} which suffices to prove \eqref{Lower bound for gm}.
	
	%$$\#(\mathcal{MQ}\cap S(\vec{a})\cap T(\vec{b}))\leq \frac{x}{6\log x},$$ 
\section{Proof of Lemma \ref{rdm}}
In this section, we show how the existence of a good sieve weight implies Lemma \ref{rdm}, following Section 6 of \cite{Fo}. Indeed, the methods are identical to those of \cite{Fo}. However, we must make some minor changes since $\sigma$ and $\mathcal{MQ}$ are different in our situation. 

Set $r:=\lfloor{\log^{1/5}x}\rfloor$ and let $(h_1,h_2,\dots,h_r)$ be an admissible $r$-tuple contained in $[0,2r^2]$: for instance, one can take $(h_1,\ldots,h_r)$ to be the first $r$ primes greater than $r$. The following lemma is the main tool for showing the existence of good choices for $(b_p)_{p\in\mathcal{P}}$.
\begin{lemma}\label{existence of good weights}
	Let $x,y$ be defined as before, and suppose $x$ is sufficiently large. Let $r$ be an integer with $$r_0\leq r\leq \log^{1/5} x,$$ for some sufficiently large absolute constant $r_0$, and let $(h_1,h_2,\dots,h_r)$ be an admissible $r$-tuple contained in $[0,2r^2]$. Then there exists a positive quantity $\tau\geq x^{-o(1)}$, a positive quantity $u$ depending only on $r$ with $u\asymp \log r$, and a non-negative weight function $w(p,n)$ defined on $\mathcal{P}\times ([-y,y]\cap \mathbb{Z})$ such that
	\begin{itemize}
		\item Uniformly for every $p\in \mathcal{P}$,
		\begin{equation}\label{nsum}
		\sum_{n\in\mathbb{Z}}w(p,n)=\left(1+O\left(\frac{1}{\log_2^{10}x}\right)\right)\tau \frac{M}{\phi(M)}\frac{y}{\log^r x}.
		\end{equation}
		\item Uniformly for every $q\in \mathcal{MQ}$ and $i=1,2,\dots,r$, 
			\begin{equation}\label{psum}
			\sum_{p\in\mathcal{P}}w(p,q-h_ip)=\left(1+O\left(\frac{1}{\log_2^{10}x}\right)\right)\tau \frac{u}{r}\frac{x}{2\log^r x}.
			\end{equation}
		\item Uniformly for all $p\in \mathcal{P}$ and $n\in \mathbb{Z}$, 
		\begin{align}\label{wpn}
w(p,n)=O\left( x^{1/3+o(1)}\right).
		\end{align}
%		\item Uniformly for every $h=O(y/x)$ that is not equal to any of the $h_i$, 
%			\begin{equation}\label{psum2}
%			\sum_{q\in \mathcal{MQ}}\sum_{p\in\mathcal{P}}w(p,q-hp)=\left(\frac{1}{\log_2^{10}x}\tau \frac{y}{\log x}\frac{x}{\log^r x}\right).
%			\end{equation}
	\end{itemize}
\end{lemma}
 The weight $w(p,n)$ can be thought of as a smoothed out indicator function of $n+h_1p,\ldots,n+h_rp$ all being ``almost in $\mathcal{MQ}$''. By ``almost in $\mathcal{MQ}$'' we mean numbers of the form $mq'$, where $m$ is defined as before and $q'$ has only large prime factors. Thus, in light of \eqref{nsum}, the weights are of size $\tau$ on average. In this section we show how Lemma \ref{existence of good weights} implies Lemma \ref{rdm}. Lemma \ref{existence of good weights} will be proved in a later section.

	Recall that $\vec{\bold{a}}=(a_s\pmod s) $ are chosen uniformly and independently for $s\in\mathcal{S}$.
\begin{lemma}\label{st}
	Let $t\leq \log x$ and let $n_1,\dots, n_t$ be distinct integers in $[-x^2,x^2]$. Then 
	\begin{equation*}
	 \mathbb{P}(n_1,\dots,n_t\in S(\vec{\bold{a}}))=\left(1+O\left(\frac{1}{\log^{16}x}\right)\right)\sigma^t
	\end{equation*}
\end{lemma}
\begin{proof} Identical to Lemma 6.1 in \cite{Fo}. 
%For each prime $s\in \mathcal{S}$, the integers $n_1,\dots,n_t$ lie in $t$ distinct residue class modulo $s$ unless $n_i\equiv n_j\mod{s}$ for some $1\leq i<j\leq t$. Since $s\geq \log^{20} x$, and $|n_i-n_j|\leq 2x^2$, the number of those $s$ is $O(t^2\frac{\log x}{\log\log^{20}x})=O(\log^3x)$. Thus for all but $O(\log ^3 x)$ choices of $s\in\mathcal{S}$, $n_1,n_2\dots,n_t$ occupy $t$ distinct residue classes modulo $s$. In this case, $$\mathbb{P}(n_1,\dots,n_t\in S(\vec{\bold{a}}))=1-\frac{t}{s},$$ otherwise  $$\mathbb{P}(n_1,\dots,n_t\in S(\vec{\bold{a}}))=\left(1+O\left(\frac{1}{\log^{19}x}\right)\right)(1-\frac{t}{s}),$$ since $t\leq \log x$. So
%\begin{align*}
%\mathbb{P}(n_1,\dots,n_t\in S(\vec{\bold{a}}))&=\left(1+O\left(\frac{1}{\log^{19}x}\right)\right)^{O(\log^3x)}\prod_{s\in\mathcal{S}}(1-\frac{t}{s})\\
%&=\left(1+O\left(\frac{1}{\log^{16}x}\right)\right)\sigma^t\prod_{s\in\mathcal{S}}\left(1+O\left(\frac{t^2}{s^2}\right)\right)\\
%&=\left(1+O\left(\frac{1}{\log^{16}x}\right)\right)\sigma^t.
%\end{align*} 
\end{proof}
Lemma \ref{st} quantifies the fact that the events $n\in S(\vec{\bold{a}})$ are almost independent for different choices of $n$. Its uniformity will be useful in what follows and is the most crucial property provided by our choices of $\vec{\bold{a}}$. For instance, the following corollary easily establishes \eqref{rdm1}.
\begin{corollary}\label{sa}
	With probability $1-O(1/\log^{6}x)$, 
	\begin{equation}
	\#\left(\mathcal{MQ}\cap S(\vec{\bold{a}})\right)=\left(1+O\left(\frac{1}{\log^{4}x}\right)\right)\sigma\#\mathcal{MQ}\leq c \frac{41}{1-\epsilon}\frac{x}{\log x}\log_2x.
	\end{equation}
\end{corollary}
\begin{proof}
From Lemma \ref{st}, 
\begin{align*}
\mathbb{E} \#(\mathcal{MQ}\cap S(\vec{\bold{a}}))&=\sum_{q\in\mathcal{MQ}}\mathbb{P}(q\in S(\vec{\bold{a}}))=\left(1+O\left(\frac{1}{\log^{16}x}\right)\right)\sigma\#\mathcal{MQ},\\
\mathbb{E}\left( \#(\mathcal{MQ}\cap S(\vec{\bold{a}}))\right)^2&=\sum_{q_1,q_2\in\mathcal{MQ}}\mathbb{P}(q_1,q_2\in S(\vec{\bold{a}}))\\
&=\left(1+O\left(\frac{1}{\log^{16}x}\right)\right)\left(\sigma\#\mathcal{MQ}+\sigma^2\#\mathcal{MQ}(\#\mathcal{MQ}-1)\right),
\end{align*}
Thus, 
\begin{align*}
\mathbb{E}\left( \#(\mathcal{MQ}\cap S(\vec{\bold{a}}))-\sigma\#\mathcal{MQ}\right)^2
=O\left(\frac{1}{\log^{16}x}\right)\left(\sigma\#\mathcal{MQ}\right)^2,
\end{align*}
and by Chebyshev's inequality  
\begin{align*}
\mathbb{P}\left( \left|\#(\mathcal{MQ}\cap S(\vec{\bold{a}}))-\sigma\#\mathcal{MQ}\right|\geq\frac{\sigma\#\mathcal{MQ}}{\log^4 x}\right)
\ll\frac{\left(\sigma\#\mathcal{MQ}\right)^2/\log^{16}x}{\sigma^2\#\mathcal{MQ}^2\log^{-8}x}\ll\frac{1}{\log^8x}.
\end{align*}
\end{proof}
	
Our next aim is to prove \eqref{rdm0}. We consider the integers in $[-y,y]$ with probability density 
$$\mathbb{P}(\tilde{\bold{n}}_p=n)=\frac{w(p,n)}{\sum_{n'}w(p,n')}.$$
For fixed $\vec{a}$ , let 
\begin{align*}X_p(\vec{a})=\sum_{n}\mathbb{P}(\tilde{\bold{n}}_p=n)\mathds{1}_{\{n+h_ip\in S(\vec{a})\forall i\}},\\
Z_p(n;\vec{a})=\mathbb{P}(\tilde{\bold{n}}_p=n)\mathds{1}_{\{n+h_ip\in S(\vec{a})\forall i\}}.
\end{align*}
Let $\mathcal{P}(\vec{a})$ denote the set of primes in $\mathcal{P}$ such that 
\begin{equation}\label{depa}
\left|X_p(\vec{a})-\sigma^r\right|\leq \frac{\sigma^r}{\log^3 x}.
\end{equation}	
For $p\in \mathcal{P}\setminus \mathcal{P}(\vec{a})$, set $\bold{n}_p=0$. For $p\in \mathcal{P}(\vec{a})$, let $\bold{n}_p$ have the conditional distribution
$$\mathbb{P}(\bold{n}_p=n|\vec{\bold{a}}=\vec{a})=\frac{Z_p(n;\vec{a})} {X_p(\vec{a})}.$$ 

\begin{lemma}\label{pa}
	With probability $1-O(1/\log^{3} x)$, $P(\vec{a})$ contains all but $O(\frac{1}{\log^{3} x}\frac{x}{\log x})$ elements of $\mathcal{P}$.
\end{lemma}
\begin{proof} The proof is identical to Lemma 6.3 in \cite{Fo}.
%By two applications of Lemma \ref{st} we have
%\begin{align*}
%\mathbb{E}\left(X_p(\vec{\bold{a}})\right)
%&=\sum_n\mathbb{P}(\tilde{\bold{n}}_p=n)\mathbb{P}(n+h_ip\in S(\vec{\bold{a}})\forall i)= \left(1+O\left(\frac{1}{\log^{16} x}\right)\right)\sigma^r, \\
%\mathbb{E}\left(X_p(\vec{\bold{a}})^2\right)&=\sum_{n_1,n_2}\mathbb{P}(\tilde{\bold{n}}_p=n_1)\mathbb{P}(\tilde{\bold{n}}_p=n_2)\mathbb{P}(n_l+h_ip\in S(\vec{\bold{a}})\forall 1\leq i\leq r;l=1,2)\\
%&=\left(1+O\left(\frac{1}{\log^{16} x}\right)\right)\sigma^{2r}+O\left(\frac{y^2}{p} \max_n(\mathbb{P}(\tilde{\bold{n}}_p=n))^2\right),\\
%\end{align*}
%since $\#\{n_l+h_ip:i=1\dots,r;l=1,2\}=2r$ unless $n_1\equiv n_2\mod p$. By our definition of $\mathbb{P}(\tilde{\bold{n}}_p=n)$, Lemma \ref{existence of good weights}, and the bound $\sigma \gg x^{-o(1)}$, we have
%\begin{align*}
%\mathbb{E}\left(|X_p(\vec{\bold{a}})-\sigma^r|^2\right)\ll \frac{\sigma^{2r}}{\log^{16}x}
%\end{align*}
%for all $p\in\mathcal{P}$ and $n\in\mathbb{Z}$.
%
%By Chebyshev's inequality, 
%$$\mathbb{P}\left(|X_p(\vec{\bold{a}})-\sigma^r|^2\geq\frac{\sigma^r}{\log^3 x}\right)\ll\frac{1}{\log^{10}x},$$ for all $p\in \mathcal{P}.$ Summing over $p\in \mathcal{P}$,
%$$\mathbb{E}(\mathcal{P}\setminus\mathcal{P(\vec{\bold{a}})})=\sum_{p\in\mathcal{P}}\mathbb{P}\left(|X_p(\vec{\bold{a}})-\sigma^r|^2\geq\frac{\sigma^r}{\log^3 x}\right)\ll \frac{x}{\log^{11}x},$$
%so by Markov's inequality 
% $$\mathbb{P}\left(\#\left(\mathcal{P}\setminus\mathcal{P}(\vec{\bold{a}})\right)\geq\frac{x}{\log^6 x}\right)\ll\frac{1}{\log^{5}x}.$$
\end{proof}

Recall that $\bold{e}_p(\vec{a})=\{\bold{n}_p+h_ip:1\leq i\leq r\}\cap\mathcal{MQ}\cap S(\vec{a}).$  Hence, for $p\in \mathcal{P}(\vec{a})$, we have 
$$\mathbb{P}(q\in\bold{e}_p(\vec{a}))=\sum_{i=1}^r\mathbb{P}(\bold{n}_p=q-h_ip|\vec{\bold{a}}=\vec{a})\leq \frac{r x^{1/3+o(1)}}{(1+\frac{1}{\log^3 x})\sigma^r y}\leq x^{-1/2-1/10},$$ which proves (\ref{rdm0}). 

It remains to prove (\ref{rdm2}). Define $C := \frac{\phi(M)}{M}\frac{ux}{2 \sigma y}$. Note that by our choice of $y$, \eqref{ck}, and \eqref{sigma}, we have
\begin{align*}
C = \frac{\phi(M)}{M}\frac{ux}{2\sigma y}&\asymp  \frac{\phi(\kappa)}{\kappa}\frac{\phi(M)}{M}\frac{\log x\log_3x}{(\log_2x)^2y}x\log_2  x \asymp \frac{1}{c}.
\end{align*}
\begin{lemma}\label{coveru}
With probability $1-O(1/\log_2^2 x)$, we have
\begin{align*}
\sigma^{-r}\sum_{p\in \mathcal{P}(\vec{\bold{a}})}\sum_{i=1}^rZ_p(q-h_ip;\vec{\bold{a}})&=\left(1+O\left(\frac{1}{\log_2^{3}x}\right)\right) C,
\end{align*}
for all but $O\left(\frac{1}{\log_2^{2}x} \#\left(\mathcal{MQ}\cap S(\vec{\bold{a}})\right)\right)$ elements of $\mathcal{MQ}\cap S(\vec{\bold{a}})$.
\end{lemma}
\begin{proof}

First we show that, with probability $1-O(1/\log x)$, replacing $\mathcal{P}(\vec{\bold{a}})$ by $\mathcal{P}$ has a negligible effect on the sum. By Lemma \ref{st}, we have 
	\begin{align*}
	\mathbb{E}\ \sigma^{-r}\sum_{n}\sum_{p\in \mathcal{P}}Z_p(n;\vec{\bold{a}})&=\sigma^{-r}\mathbb{E}\ \sum_n\sum_{p\in \mathcal{P}}\mathbb{P}(\tilde{\bold{n}}_p=n)\mathbb{P}(n+h_ip\in S(\vec{\bold{a}}),\forall i)\\
	&=\left(1+O\left(\frac{1}{\log^{16}x}\right)\right)\#\mathcal{P}.
	\end{align*}
From \eqref{depa} and Lemma \ref{pa}, we have
\begin{align*}
\mathbb{E}\ \sigma^{-r}\sum_{n}\sum_{p\in \mathcal{P}(\vec{\bold{a}})}Z_p(n;\vec{\bold{a}})&=\sigma^{-r}\mathbb{E}\sum_{p\in \mathcal{P}(\vec{\bold{a}})}\sum_n\mathbb{P}(\tilde{\bold{n}}_p=n)\mathds{1}(n+h_ip\in S(\vec{{\bold{a}}})\forall i)\\
&=\sigma^{-r}\mathbb{E}\sum_{p\in \mathcal{P}(\vec{\bold{a}})}X_p(\vec{\bold{a}})=\left(1+O\left(\frac{1}{\log^3x}\right)\right)\mathbb{E}\#\mathcal{P}(\vec{{\bold{a}}})\\
&=\left(1+O\left(\frac{1}{\log^3x}\right)\right)\#\mathcal{P}
\end{align*}	
Thus
\begin{align*}
\mathbb{E}\ \sigma^{-r}\sum_{n}\sum_{p\in\mathcal{P}\setminus \mathcal{P}(\vec{\bold{a}})}Z_p(n;\vec{\bold{a}})
&=O\left(\frac{1}{\log^3x}\#\mathcal{P}\right)=O\left(\frac{x}{\log^4x}\right). 
\end{align*}
By Markov's inequality, with probability $1-O\left(\frac{1}{\log x}\right)$,
\begin{align*}
 \sigma^{-r}\sum_{n}\sum_{p\in\mathcal{P}\setminus \mathcal{P}(\vec{\bold{a}})}Z_p(n;\vec{\bold{a}})
&=O\left(\frac{x}{\log^3x}\right). 
\end{align*} 
By restricting $n$ to $q-h_ip$ for each $i=1,\dots, r$, we see that with probability $1-O(1/\log x)$ we have
\begin{align*}
\#\Bigg\{q\in\mathcal{MQ}\cap S(\vec{\bold{a}})&:\left|\sigma^{-r}\sum_{i=1}^r\sum_{p\in\mathcal{P}\setminus \mathcal{P}(\vec{\bold{a}})}Z_p(q-h_ip;\vec{\bold{a}})\right|\geq \frac{1}{\log x}\Bigg\} \\
&\leq \log x \sum_{q \in \mathcal{MQ} \cap S(\vec{\bold{a}})} \sigma^{-r} \sum_{i=1}^r \sum_{p \in \mathcal{P}\setminus \mathcal{P}(\vec{\bold{a}})} Z_p(q - h_i p ; \vec{\bold{a}}) \\
&\ll \frac{rx}{\log^2 x} \ll \frac{x}{\log x\log_2^2x}. 
\end{align*}
%Thus, it is enough to show that with probability $1-o(1)$, for all but $O\left(\frac{x}{\log x\log_2^2x}\right)$ elements of $\mathcal{MQ}\cap S(\vec{\bold{a}})$, we have
%\begin{align*}
%\sigma^{-r}\sum_{i=1}^r\sum_{p\in\mathcal{P}}Z_p(q-h_ip;\vec{\bold{a}})
%= \left(1+O\left(\frac{1}{\log_2^{3}x}\right)\right)\frac{\phi(M)}{M}\frac{ux}{2 \sigma y}. 
%\end{align*}
By \eqref{nsum} and \eqref{psum} in Lemma \ref{existence of good weights}, we have for all $q \in \mathcal{MQ}$ and for all $1 \leq i \leq r$,
\begin{align*}
\sum_{p\in \mathcal{P}}\mathbb{P}(\tilde{\bold{n}}_p=q-h_ip)
&= \left(1+O\left(\frac{1}{\log_2^{10}x}\right)\right)\frac{C}{r}.\ 
\end{align*}
Using this relation and Lemma \ref{st}, we have
\begin{align*}
&\ \ \ \ \ \mathbb{E}\ \sum_{q\in\mathcal{MQ}\cap S(\vec{\bold{a}})}\sigma^{-r}\sum_{i=1}^r\sum_{p\in\mathcal{P}}Z_p(q-h_ip;\vec{\bold{a}})\\&=\sigma^{-r}\sum_{q\in \mathcal{MQ}}\sum_{i=1}^r\sum_{p\in \mathcal{P}}\mathbb{P}(q+(h_j-h_i)p\in S(\vec{\bold{a}})\forall j)\mathbb{P}(\tilde{\bold{n}}_p=q-h_ip)\\
&=\left(1+O\left(\frac{1}{\log^{16}x}\right)\right)\sum_{q\in \mathcal{MQ}}\sum_{i=1}^r\sum_{p\in \mathcal{P}}\mathbb{P}(\tilde{\bold{n}}_p=q-h_ip)\\
&=\left(1+O\left(\frac{1}{\log_2^{10}x}\right)\right) C\sigma\#\mathcal{MQ}.
\end{align*} 
We also compute the second moment:
\begin{align*}
&\ \ \ \ \ \mathbb{E}\ \sum_{q\in\mathcal{MQ}\cap S(\vec{\bold{a}})}\left(\sigma^{-r}\sum_{i=1}^r\sum_{p\in\mathcal{P}}Z_p(q-h_ip;\vec{\bold{a}})\right)^2\\&=\sigma^{-2r}\sum_{\substack{q\in \mathcal{MQ},p_l\in\mathcal{P}\\i_l=1\dots,r,i=1,2}}\mathbb{P}(q+(h_j-h_{i_l})p_l\in S(\vec{\bold{a}}), \forall 1\leq j\leq r,l=1,2) \prod_{l=1}^2 \mathbb{P}(\tilde{\bold{n}}_{p_l}=q-h_{i_l}p_l)\\
&=\left(1+O\left(\frac{1}{\log^{16}x}\right)\right)\sigma^{-1}\sum_{q\in \mathcal{MQ}}\left(\sum_{i=1}^r\sum_{p\in \mathcal{P}}\mathbb{P}(\tilde{\bold{n}}_p=q-h_ip)\right)^2\\
&=\left(1+O\left(\frac{1}{\log_2^{10}x}\right)\right) C^2\sigma\#\mathcal{MQ}, 
\end{align*} 
where in the second equality the contribution of the diagonal terms is negligible. From Corollary \ref{sa}, we may assume that $\#(\mathcal{MQ} \cap S(\vec{\bold{a}})) = (1 + O(1/\log^4 x)) \sigma \#\mathcal{MQ}$. Then by Markov's inequality, we have
\begin{align*}
\mathbb{P} \left( \sum_{q\in\mathcal{MQ}\cap S(\vec{\bold{a}})} \left(\sigma^{-r}\sum_{i=1}^r\sum_{p\in\mathcal{P}}Z_p(q-h_ip;\vec{\bold{a}}) - C\right)^2 \geq \frac{\sigma \#\mathcal{MQ} C^2}{\log_2^8 x} \right) \ll \frac{1}{\log_2^2 x}.
\end{align*}
Therefore, with probability $1 - O(1/\log_2^2 x)$, we have
\begin{align*}
\sum_{q\in\mathcal{MQ}\cap S(\vec{\bold{a}})} \left(\sigma^{-r}\sum_{i=1}^r\sum_{p\in\mathcal{P}}Z_p(q-h_ip;\vec{\bold{a}}) - C\right)^2 \leq \frac{\sigma \#\mathcal{MQ} C^2}{\log_2^8 x}.
\end{align*}
It follows that with probability $1- O(1/\log_2^2 x)$ we have
\begin{align*}
&\#\left\{q\in\mathcal{MQ}\cap S(\vec{\bold{a}}):\left|\sigma^{-r}\sum_{i=1}^r\sum_{p\in\mathcal{P}}Z_p(q-h_ip;\vec{\bold{a}})-C\right|\geq \frac{C}{\log_2^{3}x}\right\} \\&\ll \frac{1}{\log_2^{2}x} \sigma\#\mathcal{MQ}\ll \frac{1}{\log_2^{2}x}\#(\mathcal{MQ}\cap S(\vec{\bold{a}})),
\end{align*}
which concludes the proof.
\end{proof}
Since
\begin{align*}\sum_{p\in \mathcal{P}(\vec{a})}\mathbb{P}(q\in \bold{e}_p(\vec{a})|\vec{\bold{a}}=
\vec{a})
&=\left(1+\left(\frac{1}{\log^3x}\right)\right)\sigma^{-r}\sum_{i=1}^r\sum_{p\in\mathcal{P}(\vec{a})}Z_p(q-h_ip;\vec{a})\\
&=C+O_\leq\left(\frac{1}{\log_2^2x}\right),
\end{align*} 
we obtain \eqref{rdm2}.

\section{Proof of Lemma \ref{existence of good weights}}
We have seen that, in order to complete the proof of Theorem \ref{thm}, it suffices to prove the existence of a weight function $w(p,n)$ with the properties claimed in Lemma \ref{existence of good weights}. In this section we make some preliminary reductions in order to apply a general result of Maynard (Proposition 6.1 of \cite{Ma2}) on primes and linear forms.

The results on large prime gaps in \cite{Fo} rely on Maynard-Tao prime-detecting sieve weights. The biggest difference between the weights we use and the weights described in \cite{Fo} is the following. In \cite{Fo}, the authors use a parameter $B$ to avoid Siegel zeros and make their results effective. Here we modify $B$ and set $B = B^* M$, where $B^*$ is the parameter used to avoid Siegel zeros and $M = \prod_{p \leq \log x, p\mid k} p$, as above. We remark that $B^*$ will either be one or a prime of size $\gg \log_2 x$. Now $B$ is used not only to make Theorem \ref{thm} effective, but also to avoid giving small weight to integers of the form $q' m$, where $q'$ is prime and all the prime factors of $m$ divide $M$. We remark that now our $B$ plays a more important role, yet we give it the same notation so it aligns well with the statements of \cite{Ma2}. 

Let $\mathcal{L} = \{L_1 , \ldots , L_r\}$ be a set of distinct linear forms, $L_i(n) = a_i n + b_i$. Define the singular series of $\mathcal{L}$ to be 
\begin{align*}
\mathfrak{S}(\mathcal{L}) := \prod_{\substack{s \text{ prime} \\ s \nmid B}} \left(1 - \frac{\omega_{\mathcal{L}}(s)}{s}\right)\left(1-\frac{1}{s}\right)^{-r} ,
\end{align*}

where $$\omega_{\mathcal{L}}(s) := \#\{n \in \mathbb{Z}/s\mathbb{Z} : \prod_{i=1}^r L_i(n) \equiv 0 \pmod{s}\}.$$ Since the two sums in Lemma \ref{existence of good weights} are different, we will require two sets of linear forms. Fix a prime $p \in \mathcal{P}$ and let $\mathcal{L}_p := \{L_{p,1} , \ldots , L_{p,r}\}$, where \begin{align}\label{L1}
L_{p,i}(n) := n + h_i p,\ i=1,\dots,r.
\end{align}
Fix $q \in \mathcal{MQ}$ and $i \in \{1,\ldots,r\}$, let $\widetilde{\mathcal{L}}_{q,i} := \{\tilde{\mathcal{L}}_{q,i,1},\ldots,\tilde{\mathcal{L}}_{q,i,r}\},$ where
\begin{align}\label{L2}
\tilde{\mathcal{L}}_{q,i,j}(n) :=\left\{\begin{array}{lr}
n & \text{if } j=i,\\
(h_j-h_i)n+q & \text{if } j\not =i.
\end{array}\right.
\end{align}

For all that follows we set the level of distribution $\theta := \frac{1}{3}$.

\begin{lemma}\label{Our mega lemma}
	There exist quantities $I_r,J_r$ depending only on $r$ with
	\begin{align*}
	I_r \gg (2r\log r)^{-r}, \ \ \ \ \ \ \ \ J_r \asymp \frac{\log r}{r} I_r,
	\end{align*}
	and weights $w(p,n)$ such that the following assertions hold uniformly for $x^{\theta/10} \leq R \leq x^{\theta/3}$.
	\begin{itemize}
		\item Uniformly in $p \in \mathcal{P}$, we have 
		\begin{align}\label{psum in 7point1}
		\sum_{-y \leq n \leq y} w(p,n) = \left(1 + O \left(\frac{1}{(\log x)^{1/10}} \right) \right) \frac{B^r}{\phi(B)^r}\mathfrak{S}(\mathcal{L}_p) (2y) (\log R)^r I_r.
		\end{align}
		\item Uniformly for $q \in \mathcal{MQ}$, and $1 \leq i \leq r$, we have
		\begin{align*}
		\sum_{p \in \mathcal{P}} w(p,q - h_i p) = &\left(1 + O \left(\frac{1}{(\log x)^{1/10}} \right) \right)  \frac{B^{r-1}}{\phi(B)^{r-1}} \mathfrak{S}(\tilde{\mathcal{L}}_{q,i}) (\frac{x}{2\log x})(\log R)^{r+1} J_r \\
		&+ O \left(\frac{B^r}{\phi(B)^r} \mathfrak{S}(\tilde{\mathcal{L}}_{q,i}) x(\log R)^{r-1} I_r \right).
		\end{align*}
		\item We have the upper bound $w(p,n) \ll x^{2\theta/3+o(1)}$ for all $n \in \mathbb{Z}$, $p \in \mathcal{P}$.
	\end{itemize}
The implied constants depend at most on $\theta$.
\end{lemma}

Lemma \ref{Our mega lemma} will be proved below, basically as a direct consequence of Proposition 6.1 in \cite{Ma2}. We first show how Lemma \ref{Our mega lemma} implies Lemma \ref{existence of good weights}. We require the following result about the singular series $\mathfrak{S}(\mathcal{L}_p)$ and $\mathfrak{S}(\tilde{\mathcal{L}}_{q,i})$.
\begin{lemma}\label{singular}
	Fix $p \in \mathcal{P}$, $q \in \mathcal{MQ}$, $1\leq i\leq r$, and let $\mathfrak{S}(\mathcal{L}_p) $ and $\mathfrak{S}(\tilde{\mathcal{L}}_{q,i})$ be defined as above. Then $$\mathfrak{S}(\mathcal{L}_p) =\left(1+O\left(\frac{r}{x}\right)\right) \mathfrak{S}(\tilde{\mathcal{L}}_{q,i}).$$
\end{lemma}
\begin{proof}
	Let $\omega(s):=\#\{h_i \mod s\ \ | i=1,\dots,r\} $, and define
	$$\mathfrak{S} := \prod_{\substack{s \text{ prime} \\ s \nmid B}} \left(1 - \frac{\omega(s)}{s}\right)\left(1-\frac{1}{s}\right)^{-r}.$$
	It is sufficient to prove that 	
	\begin{align*}
	\mathfrak{S}(\mathcal{L}_p) =\left(1+O\left(\frac{r}{x}\right)\right)\mathfrak{S},\\ \mathfrak{S}(\tilde{\mathcal{L}}_{q,i})=\left(1+O\left(\frac{r}{x}\right)\right)\mathfrak{S}.
	\end{align*}
	From \eqref{L1}, we see that 
	$n+h_ip\equiv 0\pmod{s}$ if and only if $n\equiv -h_ip \pmod{s}$, thus 
	\begin{align}\label{singularp}
	\omega_{\mathcal{L}_p}(s) =\left\{\begin{array}{lr}
	\omega(s), & \text{if } s\not=p,\\
	1, & \text{if } s =p.
	\end{array}\right.
	\end{align}
	Since $p \geq \frac{x}{2}$, this shows that $	\mathfrak{S}(\mathcal{L}_p) =\left(1+O\left(\frac{r}{x}\right)\right)\mathfrak{S}$.
	
	Similarly, for the linear system \eqref{L2}, if $s\nmid q$, the solutions to $$n\prod_{j\not =i}(q+(h_j-h_i)n)\equiv 0 \pmod{s} ,$$ are $n\equiv 0\pmod{s}$ and  $n\equiv q(h_i-h_j)^{-1}\pmod{s}$ for $h_i-h_j \not \equiv 0 \pmod{s}$, which gives $\omega(s)$ solutions in total. If $s\mid q$, then  $\omega_{\tilde{\mathcal{L}}_{q,i}}(s)=1, $ since $n\equiv 0 \pmod{s}$ is the only solution. Thus  
	\begin{align}\label{singularq}
	\omega_{\tilde{\mathcal{L}}_{q,i}}(s) =\left\{\begin{array}{ll}
	\omega(s), & \text{if } s\nmid q,\\
	1 & \text{if } s|q.
	\end{array}\right.
	\end{align}
	Since $(s,B)=1$ and $q \in \mathcal{MQ}$ we have $s\geq x/4$, and so $\mathfrak{S}(\tilde{\mathcal{L}}_{q,i})=\left(1+O\left(\frac{r}{x}\right)\right)\mathfrak{S}$.
\end{proof}
Now we show how Lemma \ref{Our mega lemma} implies Lemma \ref{existence of good weights}.  Equation \eqref{wpn} in Lemma \ref{existence of good weights} follows directly from the last part Lemma \ref{Our mega lemma} and our choice of $\theta = \frac{1}{3}$.

Define a quantity $\tau$ by
\begin{align*}
\tau &:= 2 \frac{B^r}{\phi(B)^r} \frac{\phi(M)}{M} \mathfrak{S}(\mathcal{L}_p) (\log R)^r (\log x)^r I_r.
\end{align*}
We have $I_r \geq x^{-o(1)}$ by Lemma \ref{Our mega lemma}, and it is easy to check that $\mathfrak{S}(\mathcal{L}_p) \geq x^{-o(1)}$ (see Lemma 8.1 in \cite{Ma2}). By Mertens' theorem $\frac{\phi(M)}{M} \geq x^{-o(1)}$, so we deduce that $\tau\geq x^{-o(1)}$. This choice of $\tau$ then yields \eqref{nsum} by the first part of Lemma \ref{Our mega lemma}.

Define a quantity $u$ by
\begin{align*}
u := \frac{\phi(B)}{B} \frac{M}{\phi(M)} \frac{\log R}{\log x} \frac{r J_r}{2 I_r}.
\end{align*}
By the definition of $R$ we have $\frac{\log R}{\log x} \asymp 1$, and Lemma \ref{Our mega lemma} implies $\frac{r J_r}{2 I_r} \asymp \log r$. We also have the bound $\frac{\phi(B)}{B} \frac{M}{\phi(M)} \asymp 1$, since
\begin{align*}
\frac{\phi(B^*)}{B^*} = 1 + O \left( \frac{1}{\log_2 x} \right).
\end{align*}
It follows that $u \asymp \log r$. Taking these definitions of $u$ and $\tau$ and using the second part of Lemma \ref{singular}, we obtain \eqref{psum} from the second part of Lemma \ref{Our mega lemma}.

\section{Construction of Sieve Weights}
In this section we give the construction of the weights $w(p,n)$ and prove Lemma \ref{Our mega lemma}. Much of this section is similar to Sections 7 and 8 of \cite{Fo}. Additionally, we rely on definitions and concepts introduced in \cite{Ma2}. Readers acquainted with either of those papers will find this section familiar.

We observe that we cannot immediately apply the general results of Maynard \cite{Ma2} to prove Lemma \ref{Our mega lemma}, since the linear forms in \eqref{psum in 7point1} vary with $p$. Some preparatory work is therefore required.

We briefly touch upon the subject of Siegel zeros before discussing our weights $w(p,n)$. In order for our weights to have the desired properties, we will need to ``avoid'' Siegel zeros.
\begin{lemma}\label{Siegel zero lemma}
Let $Q \geq 100$. Then there exists a quantity $B^* = B_Q^*$ which is either equal to one or is a prime of size $\gg \log_2 Q$ with the property that
\begin{align*}
1 - \sigma \gg \frac{1}{\log(Q(1+|t|))}
\end{align*}
whenever $L(\sigma + it,\chi)=0$ and $\chi$ is a character with modulus $q \leq Q$ and $(q,B^*)=1$.
\end{lemma}
\begin{proof}
This is Corollary 6 of \cite{Fo} with minor changes to notation.
\end{proof}

We use this lemma below with $Q = \exp (c \sqrt{\log x})$, so that $B^*$ is either one or is a prime of size $\log_2 x \ll B^* \leq \exp (c \sqrt{\log x})$.

We define $W := \prod_{p \leq 2r^2, p \nmid B} p$. For $p$ not dividing $B$, let $a_{p,1}(\mathcal{L}) < \cdots < a_{p,\omega_{\mathcal{L}}(p)}(\mathcal{L})$ be the elements $n$ of $\{1,\ldots,p\}$ for which $p|\prod_{i=1}^r L_i(a_{p,i})$. If $p$ is also coprime to $W$, then for each $1 \leq c \leq \omega_{\mathcal{L}}(p)$, let $j_{p,c} = j_{p,c}(\mathcal{L})$ be the least element of $\{1,\ldots,r\}$ such that $p|L_{j_{p,c}}(a_{p,c}(\mathcal{L}))$.

Let $\mathcal{D}_r(\mathcal{L})$ denote the set
\begin{align*}
\mathcal{D}_r(\mathcal{L}) &:= \{(d_1,\ldots,d_r) \in \mathbb{N}^r : \mu^2(d_1\cdots d_r)=1; (d_1\cdots d_k,WB)=1; \\
&(d_j,p)=1 \text{ whenever } p \nmid WB \text{ and } j \neq j_{p,1},\ldots,j_{p,\omega_\mathcal{L}(p)}\}.
\end{align*}
We have the singular series
%\begin{align*}
%\mathfrak{S}(\mathcal{L}) &:= \prod_{p \nmid B} \left(1 - \frac{\omega_\mathcal{L}(p)}{p} \right) \left(1 - \frac{1}{p} \right)^r,
%\end{align*}
%and
\begin{align*}
\mathfrak{S}_{WB} (\mathcal{L}) &:= \prod_{p \nmid WB} \left(1 - \frac{\omega_\mathcal{L}(p)}{p} \right) \left(1 - \frac{1}{p} \right)^{-r}.
\end{align*}
Define the function $\varphi_{\omega_\mathcal{L}}(d) := \prod_{p|d} (p-\omega_\mathcal{L}(p))$, and let $R$ be a quantity of size $x^{\theta/10} \leq R \leq x^{\theta/3}$, where $0<\theta<1$ is an absolute constant. We set $F$ to be a smooth function supported on the simplex $\mathcal{R}_r := \{(x_1,\ldots,x_r) \in \mathbb{R}^r : x_i \geq 0, \sum_i x_i \leq 1\}$, and for any $(a_1,\ldots,a_r) \in \mathcal{D}_r(\mathcal{L})$ we define
\begin{align*}
y_{(a_1,\ldots,a_r)}(\mathcal{L}) &:= \frac{\mathds{1}_{\mathcal{D}_r(\mathcal{L})}(a_1,\ldots,a_r) W^r B^r}{\phi(WB)^r} \mathfrak{S}_{WB}(\mathcal{L}) F \left(\frac{\log a_1}{\log R},\ldots, \frac{\log a_r}{\log R} \right).
\end{align*}
For any $(d_1,\ldots,d_r) \in \mathcal{D}_r(\mathcal{L})$ we define
\begin{align*}
\lambda_{(d_1,\ldots,d_r)}(\mathcal{L}) &:= \mu(d_1\cdots d_r) d_1 \cdots d_r \sum_{d_i|a_i, \ \forall i} \frac{y_{(a_1,\ldots,a_r)}(\mathcal{L})}{\varphi_{\omega_\mathcal{L}}(a_1\cdots a_r)},
\end{align*}
and then define the function $w = w_{r,\mathcal{L},B,R}: \mathbb{Z} \rightarrow \mathbb{R}^+$ by
\begin{align*}
w(n) &:= \left(\sum_{d_i|L_i(n), \  \forall i} \lambda_{(d_1,\ldots,d_r)}(\mathcal{L}) \right)^2.
\end{align*}
Since $F$ is supported on $\mathcal{R}_r$ we note that $\lambda_{(d_1,\ldots,d_r)}(\mathcal{L})$ and $y_{(a_1,\ldots,a_r)}(\mathcal{L})$ are supported on
\begin{align*}
S_r(\mathcal{L}) &:= \mathcal{D}_r(\mathcal{L}) \cap \left\{(d_1,\ldots,d_r) : \prod_{i=1}^r d_i \leq R \right\}.
\end{align*}

Recall that $\{h_1,\ldots,h_r\}$ is an admissible $r$-tuple contained in $[0,2r^2]$. Set $R = (x/4)^{\theta/3}$. We define the function $w: \mathcal{P} \times \mathbb{Z} \rightarrow \mathbb{R}^+$ by 
\begin{align*}
w(p,n) := \mathds{1}_{[-y,y]}(n) w_{r,\mathcal{L}_p,B,R}(n)
\end{align*}
for $p \in \mathcal{P}$ and $n \in \mathbb{Z}$, with $\mathcal{L}_p=\{L_{p,i}, i=1,\dots,r\}$ as defined in \eqref{L1} and $w_{r,\mathcal{L}_p,B,R}$ as above. The set $\mathcal{L}_p$ is admissible since $\{h_1,\ldots,h_r\}$ is admissible. Following the proof of Lemma \ref{singular} we find that
\begin{align*}
%\mathfrak{S}(\mathcal{L}_p) &=  \left(1 + O \left(\frac{r}{x} \right) \right) \mathfrak{S}, \\
\mathfrak{S}_{BW}(\mathcal{L}_p) &= \left(1 + O \left(\frac{r}{x} \right) \right) \mathfrak{S}_{BW}
\end{align*}
uniformly in $p \in \mathcal{P}$ and  $\mathfrak{S}_{BW}$  independent of $p$. We also find that $S_r(\mathcal{L}_p)$ is independent of $p$. In fact, when $s\nmid WB$ and $s\leq R$ we have $w_{\mathcal{L}_p}(s)=r$, since $h_i\leq 2r^2< s$ and $s\not=p$.
This implies
\begin{align*}
\lambda_{(d_1,\ldots,d_r)}(\mathcal{L}_p) = \left(1 + O \left(\frac{r}{x} \right) \right) \lambda_{(d_1,\ldots,d_r)},
\end{align*}
for some $\lambda_{(d_1,\ldots,d_r)}$ independent of $p$ and where the error term is independent of $(d_1,\ldots,d_r)$.

To estimate the sums appearing in Lemma \ref{Our mega lemma}, we appeal to the results of \cite{Ma2}. In order to state these results, we require some notation and definitions.

Let $L(n) = an+b$ be a linear form, $a \neq 0$, where $a,b \in \mathbb{Z}$. Let $\mathcal{A}$ be a set of integers and $\mathscr{P}$ a set of primes. We define sets
\begin{align*}
\mathcal{A}(x) &:= \{n \in \mathcal{A} : x \leq n \leq 2x\}, \\
\mathcal{A}(x;q,a) &:= \{n \in \mathcal{A}(x) : n \equiv a (q)\}, \\
\mathscr{P}_{L,\mathcal{A}}(x) &:= L(\mathcal{A}(x)) \cap \mathscr{P}, \\
\mathscr{P}_{L,\mathcal{A}}(x;q,a) &:= L(\mathcal{A}(x;q,a)) \cap \mathscr{P}.
\end{align*}
Define $\phi_L(q):=\phi(|a|q)/\phi(|a|).$
\begin{defn}[Hypothesis 1, \cite{Ma}]\label{Hypothesis 1}
Let $x$ be a large quantity, $\mathcal{A}$ a set of integers, and $\mathcal{L} = \{L_1,\ldots,L_r\}$ a finite set of linear forms, and $B$ a natural number. We allow $\mathcal{A},\mathcal{L},r,$ and $B$ to vary with $x$. Let $0<\theta < 1$ be a fixed quantity independent of $x$, and let $\mathcal{L}'$ be a subset of $\mathcal{L}$. We say that the tuple $(\mathcal{A},\mathcal{L},\mathscr{P},B,x,\theta)$ obeys Hypothesis 1 at $\mathcal{L}'$ if we have the following three estimates:
\begin{enumerate}
\item ($\mathcal{A}(x)$ is well-distributed in arithmetic progressions) We have
\begin{align*}
\sum_{q \leq x^\theta} \max_{a} \left|\#\mathcal{A}(x;q,a) - \frac{\#\mathcal{A}(x)}{q} \right| \ll \frac{\#\mathcal{A}(x)}{(\log x)^{100r^2}}.
\end{align*}
\item ($\mathscr{P}_{L,\mathcal{A}}(x)$ is well-distributed in arithmetic progressions) For any $L \in \mathcal{L}'$ we have
\begin{align*}
\sum_{\substack{q \leq x^\theta \\ (q,B)=1}} \max_{a: (L(a),q)=1} \left|\#\mathscr{P}_{L,\mathcal{A}}(x;q,a) - \frac{\#\mathscr{P}_{L,\mathcal{A}}(x)}{\phi_L(q)} \right| \ll \frac{\#\mathscr{P}_{L,\mathcal{A}}(x)}{(\log x)^{100r^2}}.
\end{align*}
\item ($\mathcal{A}(x)$ is not too concentrated) For any $q \leq x^\theta$ and $a \in \mathbb{Z}$ we have
\begin{align*}
\#\mathcal{A}(x;q,a) \ll \frac{\#\mathcal{A}(x)}{q}.
\end{align*}
\end{enumerate}
\end{defn}
We will only need Definition \ref{Hypothesis 1} in the following special case.
\begin{lemma}\label{Hypothesis 1 satisfied}
Let $x$ be a large quantity. Then there exists a natural number $B^* \leq x$, which is either one or a prime, such that the following holds. Let $\mathcal{A} = \mathbb{Z}$, let $\mathscr{P} = \{p : p \nmid k\}$, and let $\theta = \frac{1}{3}$. Let $\mathcal{L} = \{L_1,\ldots,L_r\}$ be a finite set of linear forms $L_i(n) = a_in+b_i$ (which may depend on $x$) satisfying $r \leq \log^{1/5} x$, and $|a_i|,|b_i| \leq x^\alpha$ for some absolute constant $\alpha > 0$. Let $x/2 \leq y \leq x \log^2 x$, and let $\mathcal{L}' = \varnothing$ or $\mathcal{L}' = \{n\}$. Then $(\mathcal{A},\mathcal{L},\mathscr{P},B,y,\theta)$ obeys Hypothesis 1 at $\mathcal{L}'$ with absolute implied constants.
\end{lemma}
\begin{proof}
Parts (1) and (3) of Hypothesis 1 are straightforward to verify, so it remains to check (2). If $\mathcal{L}' = \varnothing$ then we are done, so assume $\mathcal{L}' = \{n\}$.

The set $\{p : y < p \leq 2y, p \nmid k\}$ differs from $\{p : y < p \leq 2y\}$ by a set of size $x^{o(1)}$, by our assumption on the number of distinct prime divisors of $k$. Hence
\begin{align*}
\Big|\#\mathscr{P}_{L,\mathcal{A}}(y;q,a) &- \frac{\#\mathscr{P}_{L,\mathcal{A}}(y)}{\phi_L(q)} \Big| \\
&= \left|\pi(2y;q,a)-\pi(y;q,a) - \frac{\pi(2y) - \pi(y)}{\phi(q)} \right| + O(x^{o(1)}).
\end{align*}
Using Lemma \ref{Siegel zero lemma} with $Q := \exp (c \sqrt{\log x})$ and modifying a standard proof of the Bombieri-Vinogradov theorem (as in Lemma 7.2 of \cite{Fo}, for example), we find that
\begin{align*}
\sum_{\substack{q \leq x^{1/2-\epsilon} \\ (q,B^*)=1}} \max_{a: (a,q)=1} \left|\#\mathscr{P}_{L,\mathcal{A}}(y;q,a) - \frac{\#\mathscr{P}_{L,\mathcal{A}}(y)}{\phi_L(q)} \right| &\ll y \exp(-c\sqrt{\log x}) + O(x^{1/2 - \epsilon + o(1)}) \\
&\ll \frac{x}{(\log x)^{100r^2}},
\end{align*}
as desired.
\end{proof}

We have the following theorem, which is Theorem 6 of \cite{Fo}.
\begin{theorem}\label{Maynard mega theorem}
Fix $\theta,\alpha > 0$. Then there exists a constant $C = C(\theta,\alpha)$ such that the following holds. Suppose that $(\mathcal{A},\mathcal{L},\mathscr{P},B,x,\theta)$ obeys Hypothesis 1 at some subset $\mathcal{L}'$ of $\mathcal{L}$. Write $r := \#\mathcal{L}$, and suppose that $x \geq C, B \leq x^\alpha$, and $C \leq r \leq (\log x)^{1/5}$. Moreover, assume that the coefficients $a_i,b_i$ of the linear forms $L_i(n) = a_in+b_i$ in $\mathcal{L}$ obey the bounds $|a_i|,|b_i| \leq x^\alpha$ for all $i = 1,\ldots,r$. Then there exists a smooth function $F: \mathbb{R}^r \rightarrow \mathbb{R}$ depending only on $r$ and supported on the simplex $\mathcal{R}_r$, and quantities $I_r,J_r$ depending only on $r$ with
\begin{align*}
I_r \gg (2r\log r)^{-r}, \ \ \ \ \ \ \ \ J_r \asymp \frac{\log r}{r} I_r,
\end{align*}
such that for $w(n)$ given in terms of $F$ as above, the following assertions hold uniformly for $x^{\theta/10} \leq R \leq x^{\theta/3}$.
\begin{itemize}
\item We have
\begin{align*}
\sum_{n \in \mathcal{A}(x)} w(n) = \left(1 + O \left(\frac{1}{(\log x)^{1/10}} \right) \right) \frac{B^r}{\phi(B)^r}\mathfrak{S}(\mathcal{L}) \#\mathcal{A}(x) (\log R)^r I_r.
\end{align*}
\item For any linear form $L(n) = a_Ln + b_L$ in $\mathcal{L}'$ with $a_n$ coprime to $B$ and $L(n) > R$ on $[x,2x]$ we have
\begin{align*}
\sum_{n \in \mathcal{A}(x)} \mathds{1}_{\mathscr{P}}(L(n)) w(n) = &\left(1 + O \left(\frac{1}{(\log x)^{1/10}} \right) \right) \frac{\phi(|a_L|)}{|a_L|} \frac{B^{r-1}}{\phi(B)^{r-1}} \mathfrak{S}(\mathcal{L}) \#\mathscr{P}_{L,\mathcal{A}}(x) (\log R)^{r+1} J_r \\
&+ O \left(\frac{B^r}{\phi(B)^r} \mathfrak{S}(\mathcal{L}) \#\mathcal{A}(x) (\log R)^{r-1} I_r \right).
\end{align*}
\item We have the upper bound $w(n) \ll x^{2\theta/3+o(1)}$ for all $n \in \mathbb{Z}$.
\end{itemize}
Here the implied constants depend only on $\theta,\alpha$, and the implied constants in Hypothesis 1.
\end{theorem}

Note that $B \ll x^2$, say, by the prime number theorem and the bound $B^* \leq \exp (c \sqrt{\log x})$.

We now turn to proving Lemma \ref{Our mega lemma}. The last part of that lemma follows immediately from Theorem \ref{Maynard mega theorem}. Consider the sum $\sum_n w(p,n)$ in Lemma \ref{Our mega lemma}. We have
\begin{align*}
\sum_{n \in \mathbb{Z}} w(p,n) &= \sum_{-y \leq n \leq y} w(p,n)= \sum_{n \in \mathcal{A}(2y)} w_{r,\mathcal{L}_p-3y,B,R}(n) + O(x^{1-c+o(1)})
\end{align*}
where $\mathcal{L}_p - 3y$ denotes the set of linear forms $n \rightarrow n+h_ip - 3y$, which is still admissible. We also have $\mathfrak{S}(\mathcal{L}_p-3y) = \mathfrak{S}(\mathcal{L}_p)$. We now apply the first part of Theorem \ref{Maynard mega theorem} with $x$ replaced by $2y$, $\mathcal{L}' = \varnothing$, and $\mathcal{L} = \mathcal{L}_p - 3y$, using Lemma \ref{Hypothesis 1 satisfied} to obtain Hypothesis 1. Thus
\begin{align*}
\sum_{n \in \mathbb{Z}} w(p,n) &= \left(1 + O \left(\frac{1}{(\log x)^{1/10}} \right) \right) \frac{B^r}{\phi(B)^r} \mathfrak{S}(\mathcal{L}_p) 2y (\log R)^r I_r.\end{align*}

 Fix $q \in \mathcal{MQ}$ and $i \in \{1,\ldots,r\}$, and consider the sum $\sum_p w(p,q-h_ip)$ in Lemma \ref{Our mega lemma}. Consider the linear form $\tilde{\mathcal{L}}_{q,i}$ in \eqref{L2}. Following the proof of Lemma \ref{singular}, we have
 $$\mathfrak{S}_{BW}(\tilde{\mathcal{L}}_{q,i}) = \left(1 + O \left(\frac{r}{x} \right) \right) \mathfrak{S}_{BW},$$
 and similarly 
 \begin{align*}
 \lambda_{(d_1,\ldots,d_r)}(\tilde{\mathcal{L}}_{q,i}) = \left(1 + O \left(\frac{r}{x} \right) \right) \lambda_{(d_1,\ldots,d_r)}.
 \end{align*}
This implies 
\begin{align*}
w_{r,\tilde{\mathcal{L}}_{q,i},B,R}(p) &= \left(1 + O \left(\frac{r}{x} \right) \right) w_{r,\mathcal{L}_p,B,R}(q-h_i p)
\end{align*}
whenever $p \in \mathcal{P}$ (the implicit $d_i$ summation variable on both sides is equal to 1). Thus
\begin{align*}
\sum_{p \in \mathcal{P}} w(p,q-h_ip) = \left(1 + O \left(\frac{r}{x} \right) \right) \sum_{n \in \mathcal{A}(x/2)} 1_\mathscr{P}(\tilde{\mathcal{L}}_{q,i,i}(n))w_{r,\tilde{\mathcal{L}}_{q,i},B,R}(n),
\end{align*}
which is equal to
\begin{align*}
\left(1 + O \left(\frac{1}{\log_2^{10} x} \right) \right) &\frac{B^{r-1}}{\phi(B)^{r-1}} \mathfrak{S}(\tilde{\mathcal{L}}_{q,i}) \frac{x}{2\log x} (\log R)^{r+1} J_r \\
&+ O \left(\frac{B^r}{\phi(B)^r} \mathfrak{S}(\tilde{\mathcal{L}}_{q,i}) x (\log R)^{r-1} I_r \right)
\end{align*}
by Theorem \ref{Maynard mega theorem}. This completes the proof of Lemma \ref{Our mega lemma}.

\section{Concluding Remarks}
Note that in the proof of Theorem \ref{thm} we chose $y$ as large as possible, essentially subject to the condition
\begin{align*}
\frac{\sigma}{u} \frac{M}{\phi(M)} \frac{y}{\log y} \ll \frac{x}{\log x}.
\end{align*}
Here we were able to take $u = \log r \gg \log_2 x$, in light of the results of \cite{Ma2}. Under the Hardy-Littlewood prime tuples conjecture, one could take $u = r$ rather that $u = \log r$. The Hardy-Littlewood prime tuples conjecture suggests that the number of integers $n \leq y$ such that $n + h_1 , \ldots , n + h_r$ are all prime is $\sim c \frac{y}{\log^r y}$, and so with this in mind we do not expect to be able to take $r$ too large. With this in mind, we predict that under the Hardy-Littlewood prime tuples conjecture, one might be able to show

%We remark that any admissible $r$-tuple $h_1 < \ldots < h_r$ satisfies $$h_r - h_1 \ll r \log r.$$ With the method above, we wish to capture elements $\leq y$ that are in the same residue class modulo $p \asymp x$. Thus we must have that $x r \log r \ll y $. If one assumes a uniform version of the Hardy-Littlewood prime tuples conjecture instead of using Maynard-Tao weights, then one could expect to take $u = r$ as opposed to $u \gg \log r$. Combining with the first equation, we would obtain $$  \log u \ll \frac{y }{ux}\ll \frac{\phi(M)}{M \sigma}.$$ Thus we expect, under a suitably uniform version of the Hardy-Littlewood prime tuples conjecture, one might have

\begin{align*}
P(k) \geq \phi(k) \log k \log_2^{2 - o(1)} k,
\end{align*}
which appears to be the limit of the current method. We remark that this in the same spirit as what appears in equation 1.5 of \cite{Mai}, where Maier and Pomerance considered the completely analogous problem of large gaps between primes. 

The main obstacle to further improvements and to removing the restriction on the number of prime factors of $k$ in Theorem \ref{thm} is our inability to work with prime factors larger than $\log k$. We agree with Pomerance's \cite{Po} opinion that the hardest case is when $k$ is a primorial.

We observe that the methods we use to prove Theorem \ref{thm} only identify $\log k$-rough numbers. Inserting this into the heuristic in Section \ref{Heuristic section}, one might expect that the least $\log k$-rough number in an arithmetic progression modulo $k$ has order $\asymp \phi(k) \log k \log_2 k$ (here we have used that the number of $\log k$-rough numbers less than $k$ is $\sim e^{-\gamma} k/\log_2 k$). However, we expect this estimate for the least $\log k$-rough number to be wrong, in light of what one could prove assuming a uniform prime tuples conjecture. We believe the basic reason is the strength of smooth number estimates. 
	
\appendix
\section{Numerical Data}
Here is a complete table of values of $k$ such that $P(k)/\phi(k)\log\phi(k)\log k>2-0.05$ for $k\leq 10^6$, where $R(k):=P(k)\pmod k$. We remark that our probabilistic heuristic predicts that for any $\epsilon > 0$, $P(k)/\phi(k)\log\phi(k)\log k>2-\epsilon$ infinitely often.  
\begin{align*}
\begin{array}{|ccccc|}
\hline k  & P(k) & R(k) & P(k)/\phi(k)\log\phi(k)\log k & \text{Factorization}\\\hline
4 & 5 & 1 & 2.60171 & 2^2 \\
5 & 19 & 4 & 2.12894 & 5^1 \\
6 & 7 & 1 & 2.81814 & 2^1 3^1 \\
461 & 37363 & 22 & 2.15991 & 461^1 \\
1623 & 123203 & 1478 & 2.20945 & 3^1 541^1 \\
1945 & 169937 & 722 & 1.96788 & 5^1 389^1 \\
3246 & 123203 & 3101 & 2.02004 & 2^1 3^1 541^1 \\
10948 & 642973 & 7989 & 1.96035 & 2^2 7^1 17^1 23^1 \\
23636 & 2183963 & 9451 & 2.08501 & 2^2 19^1 311^1 \\
199432 & 27361751 & 39567 & 1.98407 & 2^3 97^1 257^1 \\
297491 & 94537921 & 233274 & 2.00862 & 521^1 571^1 \\
732509 & 267676337 & 310552 & 2.00382 & 732509^1 \\
760303 & 280096127 & 304623 & 2.014 & 863^1 881^1 \\
783968 & 136749709 & 339277 & 1.99594 & 2^5 24499^1 \\
903797 & 342032531 & 397265 & 2.01678 & 739^1 1223^1 \\\hline
\end{array}
\end{align*}
Note when $k=636184$, $P(k)=56470591$ and $R(k)=486399$, whereas in \cite{Wag2}, they obtained $P(k)=116415479$ and $R(k)=629991$. We believe that they missed the prime $p=8900383$, which satisfies $8900383\equiv 629991 \pmod {636184}$. 

Here is a table of some statistics of the quantity  $P(k)/\phi(k)\log\phi(k)\log k$.
\begin{align*}
\begin{array}{|c|c|c|}
\hline P(k)/\phi(k)\log\phi(k)\log k & \text{number of $k$'s}\leq 10^6 & \text{proportion} \\\hline
>2.05 & 15& 1.5\times 10^{-5}\\
1.95\sim1.05 & 377310 &0.377\\ 
<0.5& 17 & 1.7\times 10^{-5}\\\hline
\end{array}
\end{align*}

We also provide a table of values of $k \leq 10^6$ such that $P(k)/\phi(k)\log\phi(k)\log k< 0.5$ .
\begin{align*}
\begin{array}{|ccccc|}
\hline k  & P(k) & R(k) & P(k)/\phi(k)\log\phi(k)\log k & \text{Factorization}\\\hline
44 & 113 & 25 & 0.498394 & 2^2 11^1 \\
51 & 197 & 44 & 0.45178 & 3^1 17^1 \\
75 & 293 & 68 & 0.45992 & 3^1 5^2 \\
102 & 197 & 95 & 0.384071 & 2^1 3^1 17^1 \\
105 & 419 & 104 & 0.484512 & 3^1 5^1 7^1 \\
110 & 331 & 1 & 0.477234 & 2^1 5^1 11^1 \\
130 & 389 & 129 & 0.430084 & 2^1 5^1 13^1 \\
150 & 293 & 143 & 0.396297 & 2^1 3^1 5^2 \\
198 & 643 & 49 & 0.494951 & 2^1 3^2 11^1 \\
210 & 419 & 209 & 0.421704 & 2^1 3^1 5^1 7^1 \\
228 & 761 & 77 & 0.455197 & 2^2 3^1 19^1 \\
246 & 883 & 145 & 0.457522 & 2^1 3^1 41^1 \\
312 & 1153 & 217 & 0.458184 & 2^3 3^1 13^1 \\
420 & 1201 & 361 & 0.453772 & 2^2 3^1 5^1 7^1 \\
462 & 1709 & 323 & 0.48484 & 2^1 3^1 7^1 11^1 \\
528 & 2473 & 361 & 0.48579 & 2^4 3^1 11^1 \\
570 & 2221 & 511 & 0.48907 & 2^1 3^1 5^1 19^1 \\\hline
\end{array}
\end{align*}

\section*{Acknowledgments} 
We thank Kevin Ford for making us aware of this problem, for helpful conversations and useful suggestions, and for originally introducing us to many of the techniques utilized in this work. We thank Tom\'as Silva for his helpful comments. We also thank the anonymous referee for his or her suggestions, which have improved the presentation of this paper. 

This work was completed while the second author was supported by the NSF Graduate Research Fellowship Program under Grant No. DGE-1144245. The third author received support from the NSF grant DMS-1501982.

\end{document}